\documentclass{article}

\usepackage{microtype}
\usepackage{graphicx}
\usepackage{wrapfig}
\usepackage{epsfig}
\usepackage{subcaption}

\usepackage{bm,bbm}

\usepackage{booktabs} 

\usepackage{hyperref}

\usepackage[accepted]{icml2024}

\usepackage{amsmath}
\usepackage{amssymb}
\usepackage{mathtools}
\usepackage{amsthm}
\usepackage{mathrsfs}

\usepackage[capitalize,noabbrev]{cleveref}

\theoremstyle{plain}
\newtheorem{theorem}{Theorem}[section]

\newtheorem{lemma}[theorem]{Lemma}
\newtheorem{corollary}[theorem]{Corollary}
\newtheorem{assumption}[theorem]{Assumption}
\newtheorem{remark}[theorem]{Remark}
\theoremstyle{definition}
\newtheorem{definition}[theorem]{Definition}

\newtheorem{example}{Example}

\usepackage{natbib}

\allowdisplaybreaks

\DeclarePairedDelimiter\ceil{\lceil}{\rceil}
\newcommand{\R}{\mathbb R}
\newcommand{\D}{\mathbb D}

\newcommand{\W}{\mathcal W}
\newcommand{\Ec}{\mathcal E}

\DeclareMathOperator*{\argmin}{arg\,min}
\DeclareMathOperator*{\E}{\mathbb{E}}
\newcommand{\N}{\mathbb N}
\newcommand{\A}{\mathcal A}
\newcommand{\F}{\mathcal F}
\newcommand{\M}{\mathcal M}

\newcommand{\Pb}{\mathbb P}
\newcommand{\B}{\mathbb B}
\newcommand{\z}{\mathbb S}
\newcommand{\diam}{\textup{diam}}

\newcommand{\tr}{\textup{tr}}

\usepackage{dsfont}
\newcommand{\one}{\mathds{1}}
\newcommand{\nbf}{\noindent\textbf}

\newcommand{\Sb}{\mathbb S}

\newcommand{\wbarmax}{\bar w_{\max}^{(T)}}
\newcommand{\whatmax}{\hat w_{\max}^{(T)}}
\newcommand{\T}{\mathbb T}

\usepackage[textsize=tiny]{todonotes}

\usepackage{comment}
\icmltitlerunning{Learning the Uncertainty Sets for Linear Dynamics via Set Membership: A Non-asymptotic Analysis}

\begin{document}

\twocolumn[
\icmltitle{Learning the Uncertainty Sets of Linear Control Systems via Set Membership:\\A Non-asymptotic Analysis}



\icmlsetsymbol{equal}{*}

\begin{icmlauthorlist}
\icmlauthor{Yingying Li}{equal,yyy}
\icmlauthor{Jing Yu}{equal,comp}
\icmlauthor{Lauren Conger}{comp}
\icmlauthor{Taylan Kargin}{comp}
\icmlauthor{Adam Wierman}{comp}
\end{icmlauthorlist}

\icmlaffiliation{yyy}{University of Illinois Urbana-Champaign}
\icmlaffiliation{comp}{California Institute of Technology
}

\icmlcorrespondingauthor{Yingying Li}{yl101@illinois.edu}


\vskip 0.3in
]



\printAffiliationsAndNotice{\icmlEqualContribution} 

\begin{abstract}
This paper studies uncertainty set estimation for unknown linear systems. Uncertainty sets are crucial for the quality of robust control since they directly influence the conservativeness of the control design.	Departing from the confidence region analysis of least squares estimation, this paper focuses on set membership estimation (SME). Though  good numerical performances have attracted applications of SME in the control literature, the non-asymptotic convergence rate of SME for linear systems remains an open question. This paper provides the first convergence rate bounds for SME and discusses variations of SME under relaxed assumptions. We also provide numerical results demonstrating SME's practical promise.
\end{abstract}

\section{Introduction}\label{sec: introduction}

The problem of estimating unknown linear dynamical systems of the form $x_{t+1}=A^*x_t+B^* u_t+w_t$ with unknown parameters $(A^*, B^*)$ has seen considerable progress recently \cite{sarker2023accurate, chen2021black,simchowitz2020naive,wagenmaker2020active,simchowitz2018learning,dean2018regret,abbasi2011regret}. 
Most literature focuses on the analysis of the least squares estimator (LSE) and its variants, where sharp bounds on the convergence rates for subGaussian disturbances $w_t$ have been obtained \cite{simchowitz2020naive,simchowitz2018learning}. Building on this, there is a rapidly growing body of literature on ``learning to control" unknown linear systems that leverages LSE to achieve various control objectives, such as stability and regret \cite{Chang2024regret,lale2022reinforcement,simchowitz2020naive,kargin2022thompson,mania2019certainty, dean2019safely}.

However, for successful application of learning-based control methods to safety-critical  applications,  it is crucial to quantify the uncertainties of the estimated system and to robustly satisfy safety constraints and stability despite these uncertainties \cite{wabersich2023data,brunke2022safe}. A promising framework for achieving this is to estimate the uncertainty set of the unknown system parameters and to utilize robust controllers to satisfy the robust constraints under any parameters in the uncertainty set \cite{brunke2022safe,hewing2020learning}. Uncertainty set estimation is crucial for the success of robust control: on the one hand, too large of an uncertainty set gives rise to over-conservative control actions, resulting in degraded performance; on the other hand, if the uncertainty set is underestimated and fails to contain the true system, the resulting controller may lead to unsafe behaviors \cite{brunke2022safe,petrik2019beyond}. 

To estimate uncertainty sets, a popular method is to construct LSE's confidence regions  \cite{dean2019safely,simchowitz2020naive}. However, this approach yields a confidence region for a point estimate rather than directly estimating the uncertainty set of the model. Further, the confidence regions are usually derived from concentration inequalities, which allows convergence rate analysis but may suffer conservative constant factors \cite{petrik2019beyond,simchowitz2020naive}.

In this paper, we instead focus on a direct uncertainty set estimation method: set membership estimation (SME), which estimates the uncertainty set without relying on the concentration inequalities  underlying the approaches based on LSE. SME has a long history in the control community \cite{yu2023online,lauricella2020set,lorenzen2019robust,livstone1996asymptotic,fogel1982value,bertsekas1971control}.  SME has primarily been proposed for scenarios with bounded disturbances, which is common in safety-critical systems, e.g. power systems \cite{qi2012power}, unmanned aerial vehicles (UAV) \cite{benevides2022disturbance,narendra1986robust}, and building control \cite{zhang2016decentralized}. Further, the bounded disturbance is a standard assumption in  control when certain safety requirements are desired, such as robust (adaptive) constrained control \cite{lu2023robust,lorenzen2019robust,dean2019safely}, online (constrained) control \cite{li2021online,li2021safe,liu2023online,li2023online}.

Consequently, SME has been widely adopted in the robust (adaptive) constrained control literature \cite{lorenzen2019robust,bujarbaruah2020adaptive,zhang2021trajectory, parsi2020robust, parsi2020active, sasfi2022robust} and the online control literature \cite{ho2021online, yu2023online, yeh2022robust, yu2023ltv}. Figure \ref{fig: teaser}  provides a toy example  illustrating SME's promising performance  under bounded disturbances.

On the theory side, the convergence analysis of SME generally considers a simple regression problem: 
$y_t=\theta^*x_t +w_t$ with a deterministic sequence of ${x}_t$ and bounded i.i.d.\ disturbances $w_t$ \cite{bai1998convergence,akccay2004size, kitamura2005size, bai1995membership,eising2022using}. This regression problem
does not capture the correlation between $x_t$ and the history $w_{t-1}, \dots, w_0$ in the dynamical systems.
This issue was largely overlooked in the vast literature of empirical algorithm design related to SME (for example, see \cite{lorenzen2019robust,kohler2019linear}, etc.). It is not until recently that  \cite{lu2019robust} provide the first \textit{asymptotic convergence} guarantees for SME in linear systems. However, the  \textit{non-asymptotic convergence rate} still remains open for SME in linear dynamical systems.

\begin{figure*}
	\centering
	\begin{subfigure}[b]{0.33\textwidth}
		\centering
	
        \includegraphics[width=0.9\textwidth, trim=1cm 6.5cm 2cm 7cm,clip]{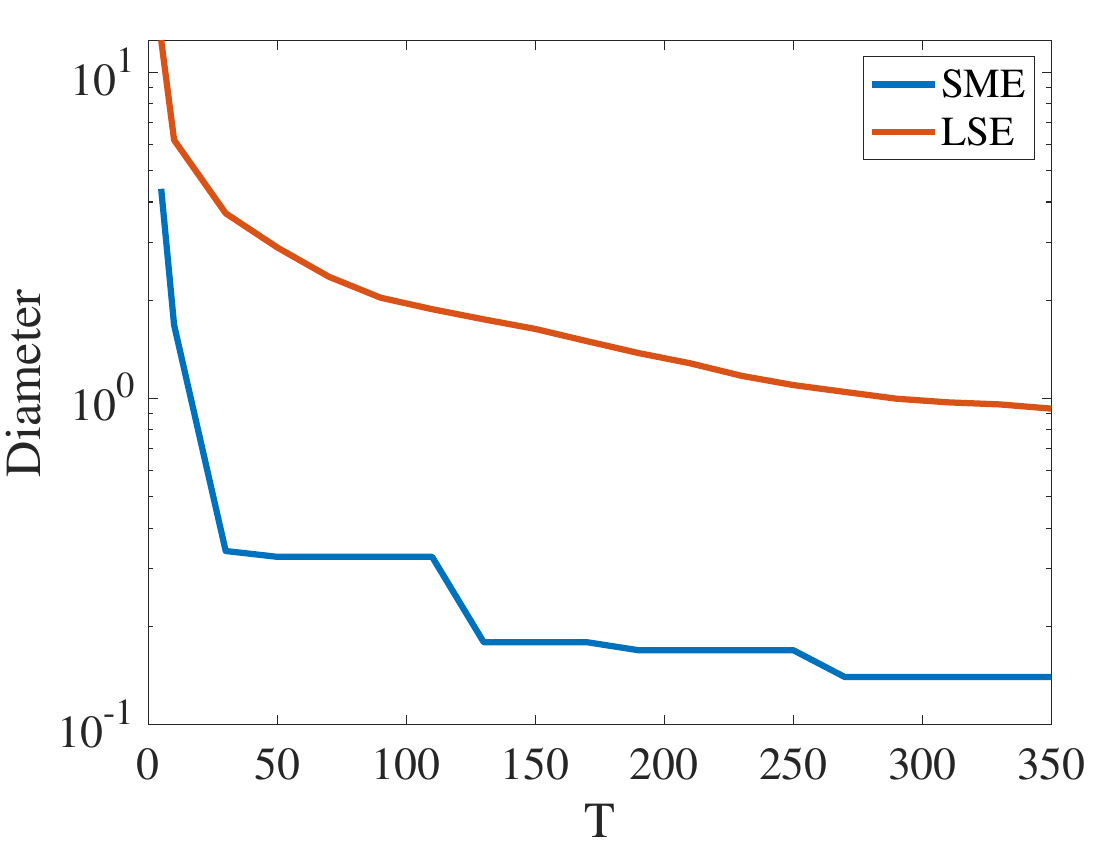}
		\caption{Uncertainty sets' diameters}
		\label{fig:diameter_d2_teaser}
	\end{subfigure}
	\hfill
	\begin{subfigure}[b]{0.33\textwidth}
		\centering
        \includegraphics[width=0.9\textwidth, trim=1cm 6.5cm 2cm 7cm,clip]{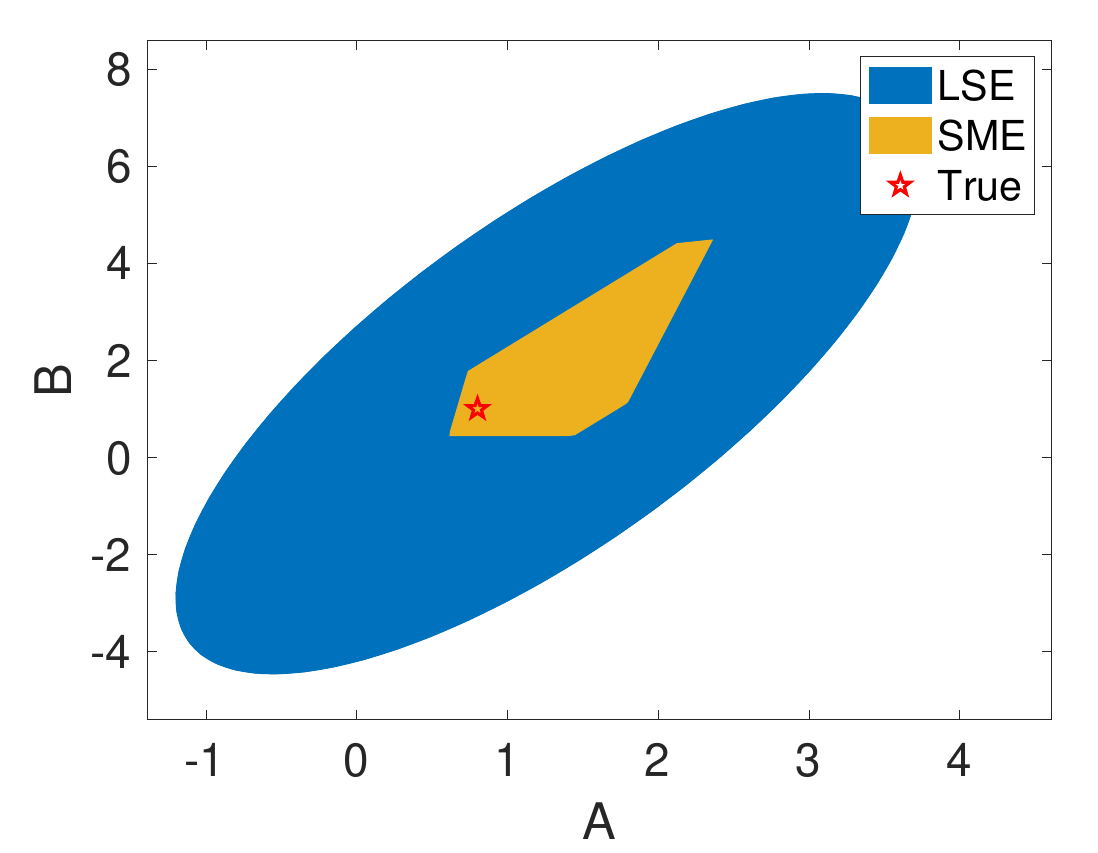}
		\caption{$T=5$}
		\label{fig:a1a3}
	\end{subfigure}
	\hfill
	\begin{subfigure}[b]{0.33\textwidth}
		\centering
        \includegraphics[width=0.9\textwidth, trim=1cm 6.5cm 2cm 7cm,clip]{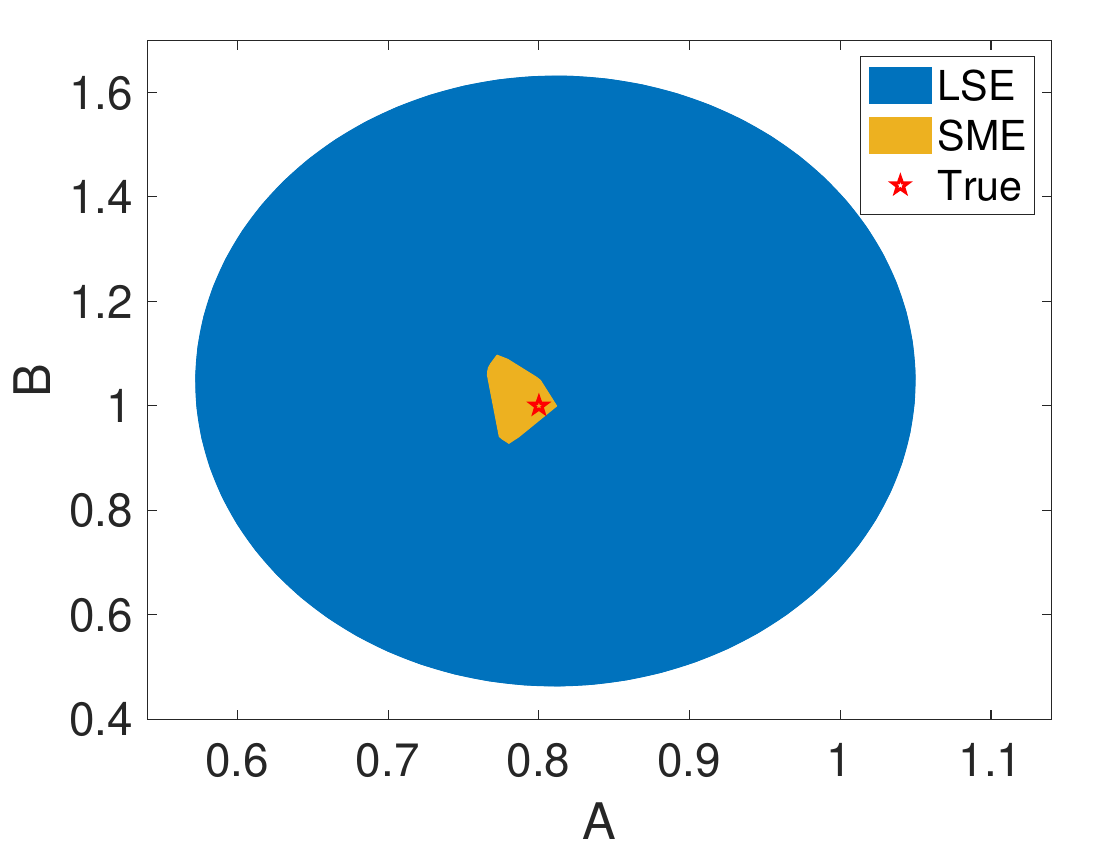}
		\caption{$T=250$}
		\label{fig:a2a4}
	\end{subfigure}
	\caption{A visualized toy example of  uncertainty set comparison between SME in \eqref{eq:membership_set} and {LSE confidence regions in \cite{simchowitz2020naive,abbasi2011regret}} for a one-dimensional system $x_{t+1}=A^*x_t +B^* u_t +w_t$, with $w_t,\, u_t \in [-1, 1]$ generated i.i.d.\ from a truncated Gaussian distribution. Detailed experiment settings are  in \Cref{append: simulation}. \nbf{Figure (a)} compares the diameters of the uncertainty sets from SME and LSE 90\% confidence bounds. \nbf{Figure (b) and (c)} visualize the the uncertainty sets  after $T=5$ and $T=250$ data points.
 }

	\label{fig: teaser}
\end{figure*}

\textbf{Contributions.} This paper tackles the open question above by providing   non-asymptotic bounds on the convergence rates of SME for  linear systems. To the best of our knowledge, this is the first convergence rate analysis of SME for dynamical systems in the literature.

We consider two scenarios in our analysis. {Firstly, when a \textit{tight} bound $\W$ on the support of $w_t$ is known,  we provide an instance-dependent convergence rate for SME. Interestingly, for several common distributions of $w_t$, SME enjoys a convergence rate $\tilde O(n_x^{1.5}(n_x+n_u)^2/T)$, which  is faster than the LSE's error bound $O(\frac{\sqrt{n_x+n_u}}{\sqrt T})$ in terms of the number of samples $T$ but is worse in terms of the dependence on state and control dimensions $n_x,n_u$. The improved convergence rate of SME with respect to $T$ is enabled by leveraging the additional boundedness property of $w_t$, which is a common assumption in robust constrained control but is not utilized in LSE's analysis.} Secondly,  when a tight bound of $w_t$ is \textit{unknown}, we introduce a UCB-SME algorithm that learns conservative upper bounds of $w_t$ from data and constructs uncertainty sets based on the conservative upper bounds. We also provide a convergence rate of UCB-SME, which has the same dependence on $T$ but has worse dependence on $n_x$ by a factor of $\sqrt{n_x}$ compared with the convergence rate with a known tight bound.

Our estimation error bound relies on a novel construction of an event sequence based on designing a sequence of stopping times. This construction, together with the BMSB condition in \cite{simchowitz2018learning}, addresses the challenge caused by the correlation between $x_t, u_t$, and the history disturbances (see the proof of Theorem \ref{thm: estimation err bdd} for more details).

Moreover, our results lay a foundation for future non-asymptotic analysis of control designs based on SME. To illustrate this, we apply  our results to  robust-adaptive model predictive control and robust system-level-synthesis (SLS) and discuss the novel non-asymptotic guarantees enabled by our convergence rates of SME.

Finally, we conduct extensive simulations to compare the numerical behaviors of SME, UCB-SME, and LSE's confidence regions, which demonstrates the  promising performance of SME and UCB-SME.

\section{Problem Formulation and Preliminaries}\label{sec: problem formulation}

\subsection{Problem Formulation} 
This paper focuses on the identification of   uncertainty sets  of  unknown system parameters in the linear dynamical system:
\begin{align}\label{equ: linear system}
	x_{t+1}=A^* x_t +B^*u_t+w_t
\end{align} 
where  $A^*, B^*$ are the unknown system parameters, $x_t \in \R^{n_x}, u_t \in \R^{n_u}$. 
For notational simplicity, we define $\theta^*=(A^*, B^*)$ by matrix concatenation and $z_t=(x_t^\top, u_t^\top)^\top \in \R^{n_z}$ by vector concatenation, where  $n_z=n_x+n_u$. Accordingly, the system \eqref{equ: linear system} can  be written as $x_{t+1}=\theta^*z_t +w_t$.

The goal of the uncertainty set 
identification problem is to determine a set $\Theta_T$ that contains the true parameters $\theta^*=(A^*, B^*)$ based on a sequence of data $\{x_t, u_t, x_{t+1}\}_{t=0}^{T-1}$. Set $\Theta_T$ is called an uncertainty set since it captures the remaining uncertainty on the system model after the revelation of the data sequence $\{x_t, u_t, x_{t+1}\}_{t=0}^{T-1}$.

Uncertainty sets play an important role in robust control, where one aims to achieve  robust constraint satisfaction \cite{lorenzen2019robust,lu2023robust},  robust objective optimization \cite{wu2013min}, and/or robust stability \cite{tu2019sample} for any model in the uncertainty set.\footnote{In addition to model uncertainties, robust control may also consider other system uncertainties, e.g., disturbances, measurement noises, etc.} 
Therefore, the diameter of the uncertainty sets heavily influences the conservativeness of robust controllers and thus the control performance. 
Formally, we define the diameter as follows. 
\begin{definition}[Diameter of a set of matrices]\label{def: set diameter}
	Consider a set $\z$ of matrices $\theta\in \R^{n_x\times n_z}$. We define the diameter of $\z$ in Frobenius norm  as $\textup{diam}(\z)=\sup_{\theta, \theta'\in \z}\|\theta-\theta'\|_F$.
\end{definition}

\subsection{Set Membership Estimation (SME)} 
In this section, we review
set membership estimation (SME), which is an uncertainty set identification method that has been studied in the control literature for decades \cite{lu2023robust,bertsekas1971control}. SME primarily focuses on systems with \textit{bounded} disturbances, i.e. $w_t \in \W$ for some bounded $\W$ for all $t\geq 0$. When $\W$ is known, SME computes an uncertainty/membership set by 
\begin{equation}  
	\label{eq:membership_set}
	\Theta_T=\bigcap_{t=0}^{T-1} \{ \hat \theta : x_{t+1} -\hat \theta z_t \in \W\}.
\end{equation}
It is straightforward to see that $\theta^*\in \Theta_T$ when $w_t\in \W$. 

The bounded disturbance assumption may seem restrictive, considering that the uncertainty set identification  based on the confidence region of LSE only requires subGaussian disturbances \cite{simchowitz2018learning}. However, in many control applications,  it is reasonable and common to assume bounded $w_t$. For example, bounded disturbances is a standard assumption in the robust constrained control literature, such as robust constrained LQR \cite{lu2023robust,lorenzen2019robust,lu2019robust,dean2019safely}, and online constrained control of linear systems \cite{liu2023online,li2021online}. This is different from unconstrained control, where unbounded subGaussian disturbances are usually considered \cite{tu2019sample}. The difference in the disturbance formulation is largely motivated by the applications: constrained control is mostly applied to safety-critical applications, where the disturbances are usually bounded. For example, in UAV and flight control, the  disturbances are mostly caused by wind gusts, and wind disturbances are bounded in practice \cite{benevides2022disturbance,narendra1986robust}. Similarly, in building thermal control, the  disturbances are caused by external heat exchanges, which are also bounded \cite{zhang2016decentralized}.

Ideally, one hopes that $\Theta_T$ converges to the singleton of the true model $\{\theta^*\}$ or at least a small neighborhood of $\theta^*$. This usually calls for additional assumptions, such as the persistent excitation property on the observed data and additional stochastic properties on $w_t$. In this paper, we consider the following assumptions to establish convergence rate bounds on the diameter of $\Theta_T$, which, to the best of our knowledge, is the first non-asymptotic guarantee of SME for linear dynamical systems. 

The first assumption formalizes the bounded disturbance assumption discussed above and introduces stochastic properties of $w_t$ for analytical purposes. 
\begin{assumption}[Bounded i.i.d.\ disturbances]\label{ass: on wt}
	The disturbances are box-constrained, $w_t\in \mathcal W:=\{w\in \R^{n_x}: \|w\|_\infty \leq w_{\max}\}$ for all $t\geq 0$.  Further, $w_t $ is i.i.d., has zero mean and positive definite covariance matrix $\Sigma_w$.
\end{assumption}

Assumption \ref{ass: on wt} is common in SME literature, e.g. \cite{akccay2004size, lu2019robust, eising2022using}.  In terms of generality, boundedness is essential for SME. The stochastic properties, such as i.i.d., zero mean, positive definite covariance, are standard in the recent learning-based control literature and allow the use of  statistical tools utilized and developed in the recent literature for non-asymptotic analysis \cite{simchowitz2018learning,li2021distributed}. Besides, it is worth mentioning that SME still works in non-stochastic settings. In particular, as long as $w_t \in \W$, even without the stochastic properties in Assumption \ref{ass: on wt}, 
the SME algorithm \eqref{eq:membership_set} still generates a valid uncertainty set that contains $\theta^*$. It is an interesting future direction to study the convergence rate of SME without assuming stochastic disturbances.

Next, we introduce the assumptions on $u_t$, which relies on the block-martingale small-ball (BMSB) condition proposed in \cite{simchowitz2018learning}. It can be shown that  the BMSB guarantees persistent excitation (PE) with high probability under proper conditions (see Proposition 2.5 in \cite{simchowitz2018learning} and \Cref{lem: bound PE c}). The PE condition requires that $z_t$ explores all directions, which is essential for system identification \cite{narendra1987persistent}.
\begin{definition}[Persistent excitation]\label{def: PE def}
    There exists $\alpha>0$ and $m \in \N+$, such that for any $t_0\geq 0$, 
    $$\frac{1}{m}\sum_{t=t_0}^{t_0+m-1}
    {\left(\begin{matrix}
        x_t\\
        u_t
    \end{matrix}\right)} (x_t^\top, u_{{t}}^\top) \succeq \alpha^2 I_{n_x+n_u}.$$
\end{definition}

\begin{definition}[BMSB \cite{simchowitz2018learning}]\label{def: bmsb}
	Consider a filtration $\{\F_t\}_{t\geq 1}$  and an $\{\F_t\}_{t\geq 1}$-adapted random process
	  $\{Z_t\}_{t\geq 1}$  in $\R^d$.
  $\{Z_t\}_{t\geq 1}$ satisfies the $(k, \Gamma_{sb},p)$-block martingale small-ball (BMSB) condition for $k>0$, a positive definite $\Gamma_{sb}$, and $0\leq p \leq 1$,  if the following  holds: for  any fixed $\lambda \in\R^d$ with $\|\lambda\|_2=1$, we have $\frac{1}{k}\sum_{i=1}^k \Pb(|\lambda^\top Z_{t+i}|\geq \sqrt{\lambda^\top\Gamma_{sb} \lambda}\mid \F_t)\geq p$  for all $t\geq 1$.
\end{definition}

The following is the assumption on $u_t$. 
\begin{assumption}[BMSB and boundedness]\label{ass: BMSB, bounded xt}
	With  filtration $\F_t=\F(w_0, \dots, w_{t-1}, z_0, \dots, z_t)$,   the $\F_t$-adapted stochastic process	$\{z_t\}_{t\geq 0}$ satisfies $(1, \sigma_{z}^2I_{n_z}, p_{z})$-BMSB for some $\sigma_z,p_z>0$. Besides, there exists $b_z\geq 0$ such that $\|z_t\|_2\leq b_z$ almost surely for all $t\geq 0$. 
\end{assumption}

Assumption \ref{ass: BMSB, bounded xt} requires $u_t$ to guarantee both BMSB and bounded $z_t$. This can be satisfied by several robust (adaptive) constrained control policies, such as robust (adaptive) model predictive control (MPC) \cite{lu2023robust,lorenzen2019robust,lu2019robust}, system level synthesis (SLS) \cite{dean2019safely}, and control barrier functions (CBF) \cite{lopez2020robust}. 
In the following,  we briefly discuss robust (adaptive) MPC as an example.  SLS and CBF can be similarly shown to satisfy Assumption \ref{ass: BMSB, bounded xt}.

\begin{example}[Robust (adaptive) MPC]\label{example: rampc} Robust   MPC is a popular method for the robust  constrained control \cite{rawlings2009model}, which aims to optimize the control objective while satisfying robust safety constraints, 
\begin{align}\label{equ: robust constraint}
z_t\in \mathbb Z_{\textup{safe}}, \,\text{where }
    x_{t+1}\!=\!\theta z_t+w_t, \forall\,\theta \!\in \!\Theta_0,  w_t\!\in \!\W,
\end{align}
where $\Theta_0$ is an initial uncertainty set known a priori, and the safety constraint $\mathbb Z_{\textup{safe}}$ is usually bounded. The robust MPC policy, denoted by $u_t=\pi_{\textup{RMPC}}(x_t; \Theta_0, \W)$,  satisfies the  constraints \eqref{equ: robust constraint} for any $\theta\in \Theta_0$. Therefore, it naturally guarantees bounded $z_t$ under the true $\theta^*$. Further, as shown in \cite{li2023non}, BMSB can be achieved by  adding a random disturbance, i.e. $u_t=\pi_{\textup{RMPC}}(x_t; \Theta_0, \W)+\eta_t$, where $\eta_t$ is  i.i.d., bounded, and has positive definite covariance. Therefore, the randomly perturbed robust MPC can satisfy Assumption \ref{ass: BMSB, bounded xt}. Robust adaptive MPC is based on the same control design, $u_t=\pi_{\textup{RMPC}}(x_t; \Theta_t, \W)$, but utilizes adaptively updated uncertainty sets $\Theta_t$. Notice that $\Theta_t$ is usually updated by SME in the literature of robust adaptive MPC \cite{lorenzen2019robust,lu2023robust,kohler2019linear}. 
\end{example}
We also note that BMSB and bounded $z_t$ with high probability are assumed in LSE literature (Theorem 2.4 \cite{simchowitz2018learning}), and bounded $z_t$ with high probability under subGaussian disturbances corresponds to bounded $z_t$ under bounded disturbances for linear systems (see bounded-input-bounded-output stability in Sec. 9 of \cite{hespanha2018linear}).

Finally, we assume that the bound $w_{\max}$ on $w_{t}$ is tight in all directions, which is common in the literature on SME  analysis \cite{bai1998convergence,akccay2004size,lu2019robust}. 
\begin{assumption}[Tight bound on $w_t$]\label{ass: tight bound on wt}
	For any $\epsilon>0$, there exists $q_w(\epsilon)>0$, such that for any $1\leq j \leq n$, we have
$$
\min(\Pb(w_t^j\leq \epsilon -w_{\max} ),\Pb(w_t^j\geq w_{\max}- \epsilon))\geq q_w(\epsilon),
$$
where $w_t^j$ denotes the $j$th entry of vector $w_t$.
Without loss of generality, we can further assume $q_w(\epsilon)$ to be non-decreasing with $\epsilon$ and $q_w(2w_{\max})=1$.\footnote{This is because $	\Pb(w_t^j\leq  \epsilon-w_{\max} )$ and $\Pb(w_t^j\geq w_{\max}- \epsilon)$ are non-decreasing with $\epsilon$, and $	\Pb(w_t^j\geq  - w_{\max})=\Pb(w_t^j\leq w_{\max})=1$ by Assumption \ref{ass: on wt}.}
\end{assumption}

In essence, Assumption \ref{ass: tight bound on wt}  requires that a \textit{hyper-cubic} $\mathcal W=\{w:\|w\|_\infty \leq w_{\max}\}$ should be \textit{tight} on the support of $w_t$ in all coordinate directions, that is, there exists a positive probability $q_w(\epsilon)$ such that $w_t$ visits an $\epsilon$-neighborhood of $w_{\max}$ and $-w_{\max}$, respectively, on all coordinates.

When the support of $w_t$ is indeed $\mathcal W=\{w:\|w\|_\infty \leq w_{\max}\}$, many common distributions enjoy $q_w(\epsilon)\geq \Omega(\epsilon)$.\footnote{The $\Omega(\cdot)$ notation is the lower bound version of $O(\cdot)$.} For example, for the uniform distribution on $\W$, we have $q_w(\epsilon)=\frac{\epsilon}{2w_{\max}}$; for the truncated Gaussian distribution with zero mean,  $\sigma_w^2 I_n$ covariance, and truncated region $\W$, we have  $q_w(\epsilon)=\frac{\epsilon}{2w_{\max}\sigma_w} \exp(\frac{-w_{\max}^2}{2\sigma_w^2})$; and for the uniform distribution on the boundary of $\W$ (a generalization of Rademacher distribution), we have $q_w(\epsilon)\geq\frac{1}{2n_x}\geq \Omega(\epsilon)$ (see \Cref{append: qw example} for more details). 

However, knowing a tight bound on the support of $w_t$ can be challenging in practice. Therefore, we  will discuss how to relax this assumption and learn a tight bound from data in \Cref{sec: unknow wmax}.

Further, the requirement of a hyper-cubic $\mathcal W$ can be restrictive because different entries of disturbances may have different magnitudes, resulting in a hyper-rectangular support that violates Assumption \ref{ass: tight bound on wt}. Our follow-up work \cite{xu2024convergence}  relaxes  this assumption and generalizes the results in this paper.

\section{Set Membership Convergence Analysis}
\label{sec:analysis}
\subsection{Convergence Rate of SME with Known $w_{\max}$}
We now present the main result (\Cref{thm: estimation err bdd}) of this paper, which is a non-asymptotic bound on the estimation error of SME given bounded i.i.d. stochastic disturbances.   
\begin{theorem}[Convergence rate of SME]\label{thm: estimation err bdd}
	For any $m>0$ any $\delta>0$, when $T>m$, we have
	\begin{align*}
		\Pb(\diam(&\Theta_T)>\delta)\leq  \underbrace{\frac{T}{m}  \tilde O(n_z^{2.5}) a_2^{n_z} \exp(-a_3 m)}_{\T_1}\\
	&	+ \underbrace{\tilde O((n_xn_z)^{2.5})a_4^{n_xn_z}\left(1-q_w\left(\frac{a_1\delta}{4\sqrt{n_x}}\right)\right)^{\ceil{T/m}} }_{\T_2(\delta)}
	\end{align*}
	where $a_1 = \frac{\sigma_{z} p_{z}}{4}$, $a_2=\frac{64 b_z^2}{\sigma_{z}^2 p_{z}^2}$, $a_3= \frac{p_{z}^2}{8}$, $a_4=\frac{4b_z \sqrt{n_x}}{a_1}$,  $p_z, \sigma_z, b_z$ are  defined in Assumption \ref{ass: BMSB, bounded xt},  $\ceil{\cdot}$ denotes the ceiling function, and $\diam(\cdot)$ is defined in Definition \ref{def: set diameter}, the factors hidden in $\tilde O(\cdot)$ are provided in Appendix \ref{append: precise upper bdd}. 
\end{theorem}

Theorem \ref{thm: estimation err bdd} provides an upper bound on the ``failure'' probability of SME, i.e., the probability that the diameter of the uncertainty set is larger than $\delta$. In this bound, $\T_1$  decays exponentially with  $m$, so for any small $\epsilon>0$, $m$ can be chosen such that $\T_1\leq \epsilon$, which indicates $m\geq O(n_z+\log T+\log(1/\epsilon))$. For any $\delta>0$, $\T_2(\delta)$ decays exponentially with the number of data points $T$ and involves a distribution-dependent function $q_w(\cdot)$, which characterizes how likely it is for $w_t$ to visit the boundary of $\W$ as defined in \Cref{ass: tight bound on wt}. To ensure  the probability upper bound in Theorem \ref{thm: estimation err bdd} to be less than 1, one can choose $m=O(\log T)$ and a large enough $T$ such that $T\geq O(m)=O(\log(T))$.  If $w_t$ is more likely to visit the boundary, (a larger $q_w(\cdot)$), then SME is less likely to generate an uncertainty set with a diameter bigger than $\delta$.

\textbf{Estimation error bounds when $q_w(\epsilon)=\Omega(\epsilon)$.} To provide intuition for $\T_2(\delta)$ and discuss the estimation error bound in \Cref{thm: estimation err bdd} more explicitly, we consider distributions satisfying $q_w(\epsilon)=\Omega(\epsilon)$ for all $\epsilon>0$. Notice that several common distributions satisfy this additional requirement, such as  uniform distribution  and  truncated Gaussian distribution as discussed after Assumption \ref{ass: tight bound on wt}. 
\begin{corollary}[Estimation error bound when $q_w(\epsilon)=\Omega(\epsilon)$]\label{cor: uniform dist est. bdd}

	For any $\epsilon>0$, let $$m\geq O(n_z+\log T+\log(1/\epsilon))$$ in the following.\footnote{A detailed formula is provided in Appendix \ref{sec:cor1}.}
	If $w_t$ is generated i.i.d. by a distribution satisfying $q_w(\epsilon)=\Omega(\epsilon)$ for all $\epsilon>0$, 
then with probability at least $1-2\epsilon$, for any $\hat \theta_T\in\Theta_T$, we have
	$$\|\hat \theta_T -\theta^*\|_F\leq \diam(\Theta_T) \leq \tilde O\left(\frac{n_x^{1.5}(n_x+n_u)^{2}}{ T}\right). $$

\end{corollary}

Corollary \ref{cor: uniform dist est. bdd} indicates that the estimation error  of any point in the uncertainty set $\Theta_T$ can be bounded by $ \tilde O\left(\frac{n_x^{1.5}(n_x+n_u)^{2}}{ T}\right)$ when $q_w(\epsilon)\geq \Omega(\epsilon)$. 

\textbf{Dynamical  systems without control inputs.} SME also applies to dynamical systems with no control inputs, i.e., $x_{t+1}=A^* x_t+w_t$, where the uncertainty set of $A^*$ can be computed by $\mathbb A_T=\bigcap_{t=0}^{T-1} \{ \hat A : \|x_{t+1} -\hat A x_t \|_\infty \leq w_{\max}\}.$ Its convergence rate can be similarly derived via the proof of Theorem \ref{thm: estimation err bdd}.

\begin{corollary}[Convergence rate with $B^*=0$ (informal)]\label{cor: estimation err bdd for B=0}
	For stable $A^*$,  for any $m>0,\delta>0,T>m$, we have
\begin{align*}
    \Pb(\diam(&\mathbb A_{ {T}})>\delta)\leq \frac{T}{m} \tilde O(n_x^{2.5}) a_2^{n_x}\exp(-a_3m)\\
		&+\tilde O(n_x^5)a_4^{n_x^2}(1-q_w(\frac{a_1\delta}{4\sqrt n_x}))^{\ceil{T/m}}  
\end{align*}

 Consequently, when $q_w(\epsilon)=\Omega(\epsilon)$, e.g. uniform or truncated Gaussian, we have
 $\diam(\mathbb A_{ {T}})\leq \tilde O(n_x^{3.5}/T)$. 
\end{corollary}

Note that \cite{simchowitz2018learning} have shown a lower bound $\Omega(\sqrt n_x/\sqrt T)$  for the estimation of linear systems with no control inputs when $w_t$ follows an (unbounded) Gaussian distribution. Interestingly, Corollary \ref{cor: estimation err bdd for B=0} reveals that, for some bounded-support distributions of $w_t$, e.g. Uniform and truncated Gaussian, SME is able to converge at a faster rate $\tilde O(1/T)$ in terms of the sample size $T$. This does not conflict with the lower bound in \cite{simchowitz2018learning} because SME's rate only holds for bounded disturbances. In fact, from \eqref{eq:membership_set}, it is straightforward to see that SME does not even converge under Gaussian disturbances. Therefore, SME is mostly useful in applications with bounded disturbances, e.g. robust constrained control, safety-critical systems, etc., while LSE's confidence regions are preferred for unbounded disturbances.

Lastly, Corollary \ref{cor: estimation err bdd for B=0} shows that SME's convergence rate has a poor dependence with respect to $n_x$: $\tilde O(n_x^{3.5})$. This is likely a proof artifact because we do not observe such poor dimension scaling in simulation (see \Cref{fig:dimension_gaussian}). It is left as future work to refine the dimension dependence. 

\subsection{SME with Unknown $w_{\max}$}\label{sec: unknow wmax} 

Next, we discuss the convergence rates of SME without knowing a tight bound $w_{\max}$ in three steps: 1) only knowing a conservative upper bound of $w_{\max}$, 2)  learning $w_{\max}$ from data, and 3) a variant of SME that converges without prior knowledge of    $w_{max}$.

\textbf{1) SME with a conservative upper bound for $w_{\max}$.} In many practical scenarios, it is easier to obtain an over-estimation of the range of the disturbances instead of a tight upper bound, i.e., $\hat w_{\max}\geq w_{\max}$. 
In this case, we can  show that the uncertainty set converges to a small neighborhood around $\theta^*$ of size $O(\sqrt{n_x}(\hat w_{\max}-w_{\max}))$ at the same convergence rate as \Cref{thm: estimation err bdd}.

\begin{theorem}[Convservative bound on $w_{\max}$]\label{thm: loose bound} When $w_{\max}$ in Assumption \ref{ass: tight bound on wt} is unknown but an upper bound $\hat w_{\max} \geq w_{\max}$ is known, consider the following SME algorithm: 
    $$\hat \Theta_T(\hat w_{\max})=\bigcap_{t=0}^{T-1} \{ \hat \theta : \|x_{t+1} -\hat \theta z_t \|_\infty\leq \hat w_{\max}\},$$

  For any $m>0$, $\delta>0$, $T>m$, we have
    \begin{align*}
\Pb(\diam(\hat\Theta_T)\!>\!\delta+a_5\sqrt{n_x}(\hat w_{\max}-w_{\max}))\leq \T_1\!+\!\T_2(\delta)
	\end{align*}
 where $a_5=\frac{4}{a_1}$, $\T_1, \T_2(\delta)$ are defined in Theorem \ref{thm: estimation err bdd}.
\end{theorem}

\textbf{2) Learning $w_{\max}$.}
When $w_{\max}$ is not accurately known, we can  try to learn it from the data. Let's first consider the   learning algorithm studied in \cite{bai1998convergence}.
\begin{align}\label{equ: define wbarmax}
	\bar w_{\max}^{(T)} &= \min_{\theta}\max_{0\leq t \leq T-1} \|x_{t+1}-\theta z_t\|_\infty.
\end{align}
{Though algorithm \eqref{equ: define wbarmax} cannot provide an upper bound on $w_{\max}$ under finite samples because $\bar w_{\max}^{(T)}\leq w_{\max}$ for finite $T$,\footnote{If SME does not use an upper bound on $w_{\max}$, the generated uncertainty set may not contain the true parameter $\theta^*$.} it can be shown that $\bar w_{\max}^{(T)}$ converges to $w_{\max}$ as $T\to +\infty$. The convergence for linear regression has been established in  \cite{bai1998convergence}. The following theorem establishes the  convergence and convergence rate of  algorithm \eqref{equ: define wbarmax} for linear dynamical systems. Based on this convergence rate, we will design an online learning algorithm \eqref{equ: ucb-sme}  that generates converging upper bounds of $w_{\max}$.}

\begin{theorem}\label{thm: convergence rate of wbarmax}
The  estimation $\wbarmax$ of $w_{\max}$ satisfies: 
$$0 \!\leq w_{\max}-\wbarmax \!\leq  \underbrace{b_z \diam(\Theta_T)}_{\T_3}+\underbrace{w_{\max}\!\!-\!\!\max_{0\leq t \leq T\!-\!1} \|w_t \|_\infty}_{\T_4}$$
Therefore, 
for any $\delta>0$,
    $$ \Pb(w_{\max}-\bar w_{\max}^{(T)}>\delta)\leq \T_1+\T_2\left(\frac{\delta}{2b_z}\right)+\T_5\left(\frac{\delta}{2}\right),$$
    where $\T_5(\delta)= (1-q_w(\delta))^T$.

\end{theorem}
Notice that $\T_4$ is the smallest possible learning error of $w_{\max}$ from history $w_t$, which can be achieved if one can directly measure $w_t$. However, with unknown $\theta^*$, it is challenging to measure/compute $w_t$ exactly, then 
Theorem \ref{thm: convergence rate of wbarmax} shows that the learning error of $w_{\max}$ has an additional term $\T_3$ that depends on the  uncertainty around $\theta^*$. Therefore, the convergence rate of $\wbarmax$ can be obtained by our non-asymptotic analysis of SME in Theorem \ref{thm: estimation err bdd}.

Further, when $q_w(\epsilon)=\Omega(\epsilon)$, the convergence rate of $\wbarmax$ can be explicitly bounded by $\tilde O(n_x^{1.5}n_z^2/T)$, which is of the same order as the convergence rate of the diameter of $\Theta_T$.

\begin{corollary}\label{cor: wbarmax bdd when qw=O(epsilon)}
For any $0<\epsilon<1/3$ and any $T\geq 1$, there exists $\delta_T>0$ satisfying $\lim_{T\to \infty}\delta_T=0$  such that 
$$ 0\leq w_{\max}-\bar w_{\max}^{(T)} \leq \delta_T$$
with probability at least $1-3\epsilon$.

    In particular, when $q_w(\delta)=O(\delta)$, with probability $1-3\epsilon$,
    $$ 0\leq w_{\max}-\bar w_{\max}^{(T)} \leq \delta_T=\tilde O(n_x^{1.5}n_z^2/T)$$
\end{corollary}

\textbf{3) SME with unknown $w_{\max}$.} Unfortunately, $\wbarmax$ cannot be directly applied to SME because $\wbarmax\leq w_{\max}$, which may cause $\theta^*\not \in\hat \Theta_T(\wbarmax)$. However, by leveraging our convergence rate bound in Theorem \ref{thm: convergence rate of wbarmax}, we can construct an upper confidence bound (UCB) of $w_{\max}$ and a corresponding UCB-SME  algorithm:
\begin{align}\label{equ: ucb-sme}
\whatmax&=\wbarmax+\delta_T, \quad \hat \Theta_T^{\text{ucb}}=\hat \Theta_T(\whatmax),
\end{align}
where $\delta_T$ is defined in Corollary \ref{cor: wbarmax bdd when qw=O(epsilon)}. 

Then, by combining Theorem \ref{thm: loose bound}
 and Corollary \ref{cor: wbarmax bdd when qw=O(epsilon)}, we can verify the well-definedness of UCB-SME and obtain its convergence rate.
 \begin{theorem}\label{thm: convergence of ucb-sme}
     For any $0< \epsilon <1/3$, any $T\geq 1$, with probability at least $1-3\epsilon$, we have
     \begin{align*}
         \theta^*& \in \hat \Theta^{\textup{ucb}}_T, \quad \diam(\hat \Theta^{\textup{ucb}}_T)\leq O(\sqrt{n_x}\delta_T).
     \end{align*}
     In particular, if $q_w(\epsilon)=\Omega(\epsilon)$, then $\diam(\hat \Theta^{\textup{ucb}}_T)\leq O(n_x^2n_z^2/T)$ with probability at least $1-3\epsilon$.
 \end{theorem}
 Notice that UCB-SME converges at the same rate in terms of $T$ but $\sqrt{n_x}$-worse  in terms of dimensionality when compared with SME knowing a tight bound $w_{\max}$.

\begin{remark}[Computation complexity]
    SME  can be computed by linear programming since all constraints are linear in \eqref{eq:membership_set}. Further, UCB-SME can also be computed by linear programming because \eqref{equ: define wbarmax} can be reformulated as a linear program.   However, the number of constraints for SME and UCB-SME increases linearly with $T$. To address the computation issue of SME, many computationally efficient algorithms have been proposed based on approximations of \eqref{eq:membership_set},  e.g. \citep{lu2019robust, yeh2022robust, bai1995membership}. The convergence rates of these approximate algorithms are unknown and how to design computationally efficient UCB-SME remains open.

\end{remark}

\section{Proof Sketch of Theorem \ref{thm: estimation err bdd}}\label{sec: proof sketch Thm1}

The major technical novelty of this paper is the proof of Theorem \ref{thm: estimation err bdd}, thus we describe the key ideas here.  
The complete proof  is provided in  Appendix \ref{appendix:proof_main}. For ease of notation and without loss of generality, we assume $T/m$ is an integer in the following.

Specifically, we first define a set $\Gamma_T$ on the model estimation error $\gamma=\hat \theta-\theta^*$ by leveraging the observation that $x_{s+1}-\hat \theta z_s=w_s-(\hat \theta -\theta^*) z_s$,
\begin{align}\label{equ: define Gamma t}
	\Gamma_t=\bigcap_{s=0}^{t-1} \{\gamma : \|w_s -\gamma  z_s \|_\infty \leq w_{\max}\}, \quad \forall\, t\geq 0.
\end{align}
Notice that $\Theta_t =\theta_*+\Gamma_t$, so $\diam(\Theta_t)=\diam(\Gamma_t)$, and $$\diam(\Gamma_t)=\sup_{\gamma,\gamma'\in \Gamma_t} \|\gamma-\gamma'\|_F\leq 2\sup_{\gamma\in \Gamma_t}\|\gamma\|_F.$$ Thus, we can define $\Ec_1\coloneqq\{\exists\, \gamma \in \Gamma_T, \text{ s.t. } \|\gamma\|_F\geq \frac{\delta}{2} \}$ such that ${\Pb}(\diam(\Theta_T)>\delta) \leq  {\Pb}(\Ec_1)$.

\begin{figure*}[ht]
	\centering

	\begin{subfigure}[b]{0.32\textwidth}
		\centering
		\includegraphics[width=\textwidth]{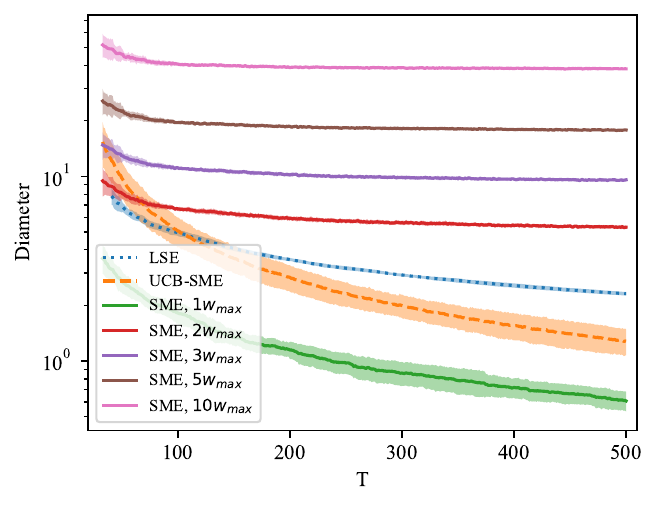}
		\caption{{Truncated Gaussian}}
		\label{fig:loose_gaussian}
	\end{subfigure}
	\hfill
	\begin{subfigure}[b]{0.32\textwidth}
		\centering
		\includegraphics[width=\textwidth]{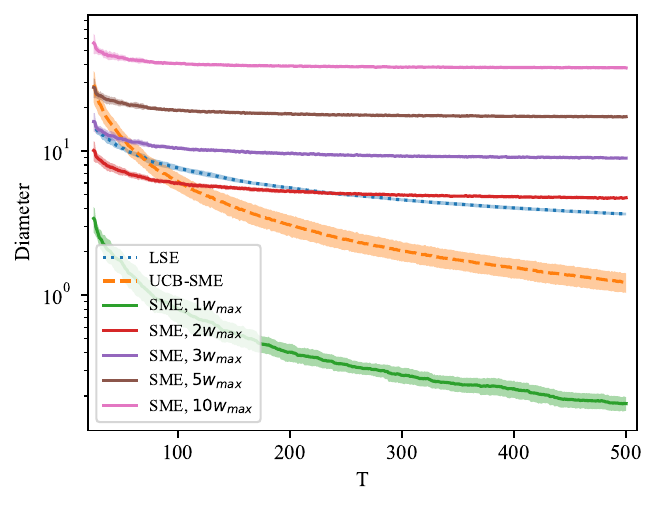}
		\caption{Uniform Distribution}
		\label{fig:loose_uniform}
	\end{subfigure}
\hfill
 	\begin{subfigure}[b]{0.35\textwidth}
	\centering
	\includegraphics[width=\textwidth, trim=0cm 0.25cm 0cm 0.26cm, clip]{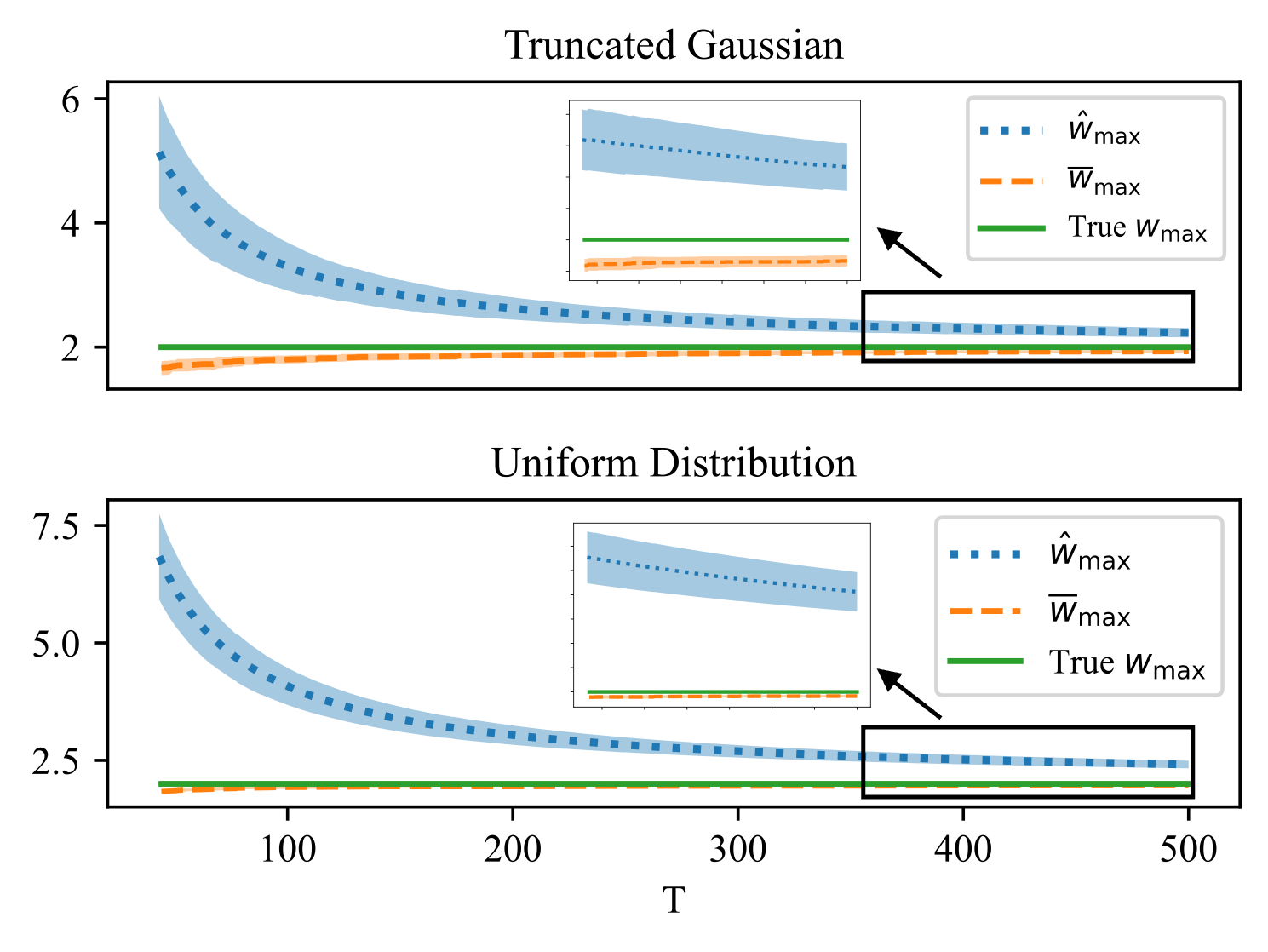}
	\caption{$\whatmax$ in UCB-SME and estimation $\wbarmax$}
	\label{fig:learning_wmax}
\end{subfigure}

	\caption{ \nbf{Figures (a)-(b)} compares the diameters of SME, UCB-SME, and SME with loose disturbance upper bounds that are 2, 3, 5, and 10 times larger than the true disturbance bound $w_{\max}$, as well as the baseline uncertainty set from the 90\% confidence region of LSE. \nbf{Figure (c)} shows the convergence  to the true bound $w_{\max}$ of the lower estimation $\bar w_{\max}$ in \eqref{equ: define wbarmax}  and the UCB  $\hat{w}_{\max}$ generated by the UCB-SME algorithm in \nbf{Figures (a)-(b)}.}
	\label{fig:loose_bound}
\end{figure*}

Next, we define an event $\Ec_2$ below, which is essentially PE on every time segments $km+1\leq t \leq km+m$ for $k\geq 0$, where the choice of $m$ will be specified later.
\begin{align*}
\Ec_2=	\left\{\frac{1}{m}\sum_{s= {1}}^{ {m}}  {z}_{km+s}  {z}_{km+s}^\top \succeq  a_1^2 I_{ {n_z}}, \forall\, 0\leq \!k\! \leq \!\!\left\lceil\frac{T}{m}\right\rceil\!\!-\!1 \right\}
\end{align*}

where $a_1 =\frac{\sigma_{z} p_{z}}{4}$. Now, by dividing the event $\Ec_1$ based on $\Ec_2$, we obtain
$$  {\Pb}(\diam(\Theta_T)>\delta)\leq \Pb(\Ec_1)\leq  {\Pb}(\Ec_2^c)+  {\Pb}( \Ec_1\cap \Ec_2 ).$$
The proof can be completed by establishing the following bounds on $ {\Pb}(\Ec_2^c)$ and $  {\Pb}( \Ec_1\cap \Ec_2 )$.
\begin{lemma}[Bound on $ {\Pb}(\Ec_2^c)$]\label{lem: bound PE c}
	$ \Pb(\Ec_2^c) \leq \T_1$,	
where $a_2=\frac{64  {b_z^2}}{\sigma_{z}^2 p_{z}^2}$ and $a_3= \frac{p_{z}^2}{8}$.
\end{lemma}

\begin{lemma}[Bound on $  {\Pb}( \Ec_1\cap \Ec_2 )$]\label{lem: bound PE and gamma>delta/2}
	$
		 {\Pb}( \Ec_1\cap \Ec_2 ) \leq \T_2(\delta)$, 
 where $a_4=\max(1, 4b_z\sqrt{n_x}/a_1)$.
\end{lemma}

Roughly, Lemma \ref{lem: bound PE c} indicates that PE holds with high probability, which is proved by leveraging the BMSB assumption and set discretization. 
The proof of Lemma \ref{lem: bound PE and gamma>delta/2} is more involved and is our major technical contribution. On a high level, the proof relies on  two technical lemmas below.
\begin{lemma}[Discretization of $\Ec_1\cap\Ec_2$ (informal)]\label{lem: discretize E1 E2}
	Let $\M=\{\gamma_1, \dots, \gamma_{v_\gamma}\}$ denote an $\epsilon_\gamma$-net of $\{\gamma: \|\gamma\|_F=1\}$. Under a proper choice of $\epsilon_\gamma$, we have $v_\gamma = \tilde O( {n_x^{2.5}}n_z^{2.5})a_4^{n_x n_z}$.\footnote{The exact formulas of $v_{\gamma}$ and $\epsilon_{\gamma}$ are in Lemma \ref{lem: mesh on unit sphere of matrices}.}
We can construct $\tilde \Gamma_T$ such that 
	\begin{align*}
		{\Pb}(\Ec_1\cap\Ec_2)&\leq  {\Pb}(\{\exists\,  1\leq i \leq v_\gamma, d\geq 0, \textup{ s.t. }d\gamma_i \in \tilde \Gamma_T\}\cap \Ec_2)\\
		& \leq \sum_{i=1}^{v_\gamma}  {\Pb}(\Ec_{1, {i}}\cap \Ec_2)
	\end{align*} 
	where $\Ec_{1, {i}}=\{\exists\,  d\geq 0, \textup{ s.t. }d\gamma_i \in \tilde \Gamma_T\}$. 
	
\end{lemma}
Lemma \ref{lem: discretize E1 E2} leverages finite set discretization to bound the existence of a feasible element in an infinite continuous set. The formal version of Lemma \ref{lem: discretize E1 E2} is provided as Lemma \ref{lem: discretize E1 E2 formal} in the appendix. 

\begin{lemma}[{Construction} of event $G_{i,k}$ via stopping times (informal)]\label{lem: construct Gik and Lik}
	Consider $\F_t$ as defined in Assumption \ref{ass: BMSB, bounded xt}. Under the conditions in Lemma \ref{lem: discretize E1 E2}, we  construct $G_{i,k}$ for all $i$ and all $0\leq k \leq T/m-1$ by
	\begin{align*}
		G_{i,k}=&\left\{\ b_{i, km+L_{i,k}} w_{km+L_{i,k}}^{j_{i, km+L_{i,k}}} \geq \frac{a_1\delta}{4\sqrt{n_x}}-w_{\max},   \textup{ and } \right.\\
		& \  \ \ \left.\frac{1}{m}\sum_{s= {1}}^{ {m}}  {z}_{km+s}  {z}_{km+s}^\top \succeq  a_1^2 I_{ {n_z}}\right\}.
	\end{align*}
where $b_{i,t}, j_{i,t}$ are measurable in $\F_t$, and $L_{i,k}$ is  constructed as a stopping time with respect to $\{\F_{km+l}\}_{l\geq 0}$. The formal definitions of $b_{i,t}, j_{i,t}, L_{i,k}$ are provided in Appendix \ref{append: define bi ji Li}.

Then, we have
	\begin{align*}
		{\Pb}&(\Ec_{1,i}\cap \Ec_2)\!\leq \!\Pb\left(\bigcap_{k=0}^{T/m-1}G_{i,k}\!\right)
\!\leq \!\left( 1-q_w\left(\frac{a_1 \delta}{4\sqrt n_x}\right)\!\right)^{\frac{T}{m}}
	\end{align*} 
 \vspace{-0.1in}
\end{lemma}
 \vspace{-0.1in}
The constructions of $G_{i,k}$ and $L_{i,k}$ in Lemma \ref{lem: construct Gik and Lik} are our major technical contribution. With the constructions above, the proof can be completed by  leveraging the conditional independence property of stopping times,  which  is briefly discussed below. Notice that by conditioning on the event $\{L_{i,k}=l\}$, we have $ w_{km+L_{i,k}}=w_{km+l}$ and $w_{km+l}$  is independent of  $ \F_{km+l}$. Consequently,  $w_{km+l}$ is also independent of $b_{i, km+L_{i,k}}, j_{i, km+L_{i,k}} $ conditioning on $\{L_{i,k}=l\}$ since $b_{i, km+l}, j_{i, km+l} $ are measurable in $\F_{km+l}$. Therefore, the probability of $G_{i,k}$ conditioning on $\{L_{i,k}=l\}$ can be bounded by the probability distribution of $w_t$, which enjoys good properties such as Assumption \ref{ass: tight bound on wt}. More details of the proof are in Appendix \ref{sec:D33}.

In conclusion, 
Lemma \ref{lem: bound PE and gamma>delta/2} follows directly from Lemma \ref{lem: discretize E1 E2} and Lemma \ref{lem: construct Gik and Lik}. Combining Lemma \ref{lem: bound PE and gamma>delta/2} and Lemma \ref{lem: bound PE c} completes the proof of Theorem \ref{thm: estimation err bdd}.

\begin{remark}[Convergence rate of SME for general time series] Similar to Theorem 2.4 in \cite{simchowitz2018learning}, our results for linear dynamical systems can also be generalized to general time series with linear responses:
$$y_t =\theta^* z_t+w_t, \quad t\geq 0,$$
where $\F_t^y=\F(w_0, \dots, w_{t}, z_0, \dots, z_t)$, $y_t\in \R^{n_y}$ is measurable in $\F^y_t$ but not in $\F^y_{t-1}$. The SME algorithm is 
$$\Theta_T^y=\bigcap_{t=0}^{T-1}\{\hat \theta: y_t-\hat \theta z_t\in\W\}.$$ Under Assumptions \ref{ass: on wt}, \ref{ass: BMSB, bounded xt}, and \ref{ass: tight bound on wt}, we have
\begin{align*}
		\Pb(\diam(&\Theta^y_T)>\delta)\leq  {\frac{T}{m}  \tilde O(n_z^{2.5}) a_2^{n_z} \exp(-a_3 m)}\\
	&	+ {\tilde O((n_yn_z)^{2.5})a_4^{n_yn_z}\left(1-q_w\left(\frac{a_1\delta}{4\sqrt{n_y}}\right)\right)^{\ceil{T/m}} },
	\end{align*}
 where $a_1,a_2, a_3$ are defined in Theorem \ref{thm: estimation err bdd} and $a_4=\frac{4b_z \sqrt{n_y}}{a_1}$.

\end{remark}

\section{Applications to Robust Adaptive  Control}
\label{sec:application}

Robust adaptive control usually involves two steps: updating the uncertainty set estimation, and designing robust controllers based on the updated uncertainty set. SME can be naturally applied to robust adaptive control as the updating rule of the uncertainty set estimation. To illustrate this, we discuss the applications of SME to two popular controllers, robust adaptive MPC and robust SLS. We focus on the implications of our convergence rates. 

\textbf{Application of SME to robust adaptive MPC.} SME has long been adopted in the robust adaptive MPC design (see e.g., \cite{kohler2019linear,lorenzen2019robust,lu2023robust}). Despite the regret analysis for unconstrained MPC and its variants (e.g. \cite{zhang2021regret,li2019online}), the non-asymptotic analysis for robust adaptive MPC remains unsolved.  
Applying \Cref{thm: estimation err bdd} straightforwardly, we can obtain a non-asymptotic estimation error bound  for robust adaptive MPC below, which lays a foundation for future regret analysis. For simplicity, we consider a tight bound $\W$ is known below, but our results for unknown $\W$ can also be applied similarly.


\begin{corollary}
Consider the robust adaptive MPC controller introduced in Example \ref{example: rampc}, where $\Theta_t$ is updated by SME and $\W$ is known.\footnote{When $\W$ is unknown, Theorems \ref{thm: loose bound}-\ref{thm: convergence of ucb-sme} all apply.} Under the conditions of Corollary \ref{cor: uniform dist est. bdd}, the estimation error for any $\hat \theta_T \in \Theta_T$ can be bounded by  $\|\hat \theta_T-\theta^*\|_F \leq \tilde O(\frac{n_x^{1.5}n_z^2}{T})$ with high probability. 
 \label{cor:rmpc}
\end{corollary}

\textbf{Application of SME to robust SLS.} Robust SLS has been proposed in \cite{dean2019safely} for robust constrained control under system uncertainties  \cite{dean2019safely}. Since \cite{dean2019safely} assumes bounded disturbances, one can apply SME for the uncertainty set estimation in place of the LSE's confidence regions in \cite{dean2019safely}. Then, by leveraging Theorems 3.1, 4.1 in \cite{dean2019safely} and our Theorem \ref{thm: estimation err bdd}, we can directly obtain a non-asymptotic sub-optimality gap for learning-based robust SLS with SME as the uncertainty set estimation. For simplicity, we consider a known tight bound $\W$, but our results for unknown $\W$ can also be similarly applied here.
\begin{corollary}
    Under the conditions in Theorem 3.1 in \cite{dean2019safely} and Corollary \ref{cor: uniform dist est. bdd}, for large enough $T$, we have  $ \frac{J(A^*, B^*, \hat{\mathbf K})- J^*}{J^*} \leq \tilde O(n_x^{1.5}n_z^2/T)$,   where $\hat{\mathbf K}$  denotes the robust SLS controller in \cite{dean2019safely} under the uncertainty set $\Theta_T$ constructed by SME, $J(A^*, B^*, \hat{\mathbf K})=\lim_{T\to +\infty}\E\frac{1}{T} \sum_{t=0}^{T-1}(x_t^\top Q x_t + u_t^\top R u_t)$ denotes the infinite-horizon averaged total cost by implementing the robust SLS controller $\hat{\mathbf K}$, and $J^*$ denotes the optimal infinite-horizon averaged total cost.
   
\end{corollary}

\vspace{-0.1in}
\section{{Numerical Experiments}}\label{sec: numerical experiments}

We evaluate the empirical performance of SME on various systems and applications. For all experiments, we use the $90\%$ confidence regions of LSE computed by Lemma E.3 in \cite{simchowitz2020naive} and Theorem 1 in \cite{abbasi2011regret} as the baseline. The details of the simulation settings are provided in Appendix \ref{append: simulation}.\footnote{The code to reproduce all the experimental results can be found at \url{https://github.com/jy-cds/non-asymptotic-set-membership}.}

\paragraph{Comparison of SME, SME with loose bound, UCB-SME, and LSE.}  This experiment is  based on the linearized longitudinal flight control dynamics of Boeing 747 as studied in recent literature on learning-based control of linear systems \cite{lale2022reinforcement, mete2022augmented}.

In \Cref{fig:loose_bound}, we show the diameters of SME, SME with loose disturbance bounds, and UCB-SME on the identification problem of the Boeing 747 dynamics with i.i.d. truncated Gaussian (\Cref{fig:loose_gaussian}) and uniform (\Cref{fig:loose_uniform}) disturbances. We use control actions sampled from a uniform distribution in both cases. In \Cref{fig:learning_wmax}, we show that both  the upper bound $\hat w_{\max}$ used for UCB-SME and the lower bound $\bar w_{\max}$ in \eqref{equ: define wbarmax} converge to the true bound $w_{\max}$ as $T$ increases.
The quantitative behaviors of SME and its variants are consistent with those predicted by our theoretical results. {In particular, in  \Cref{fig:loose_gaussian} and \Cref{fig:loose_uniform},  SME and UCB-SME outperform the $90\%$ confidence regions of  LSE in both the magnitude and the convergence rate. In \Cref{fig:learning_wmax}, we verify that the UCB estimation $\hat w^{(T)}_{\max}$ converges to the true disturbance bound $w_{\max}$ from above, while the estimation $\bar w^{(T)}_{\max}$ converges from below. It is worth noting that $\bar w^{(T)}_{\max}$ converges to $w_{\max}$ very quickly in the  simulations, allowing $\bar w^{(T)}_{\max}$ to be another potential approximation of $w_{\max}$ for SME when $T$ is very large.}

\begin{figure}
	\centering
\includegraphics[width=0.45\textwidth,trim=0cm 0.1cm 0cm 0.1cm,clip]{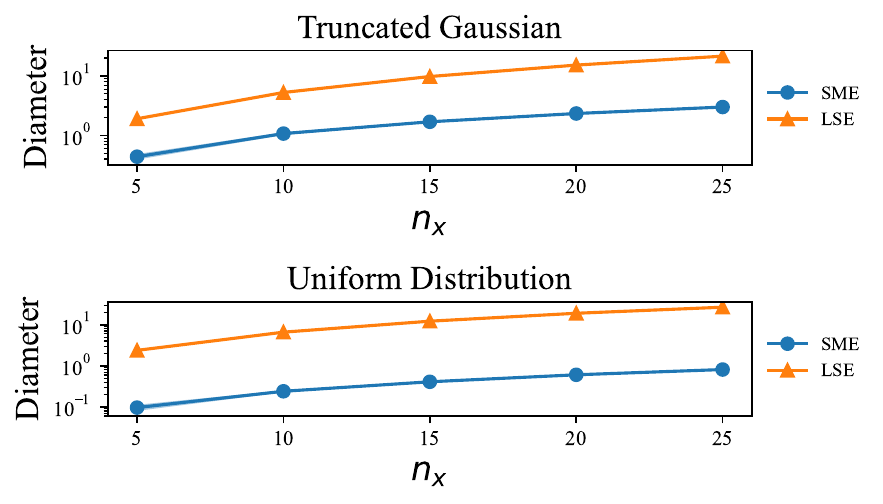}
\caption{{Diameters of the uncertainty sets constructed by SME, UCB-SME, and LSE for systems with different  dimensions.}}
		\label{fig:dimension_gaussian}
\end{figure}

\vspace{-0.12in}
\paragraph{Scaling with dimension.}
We compare the scaling of SME, SME-UCB, and LSE with respect to the system dimensions in \Cref{fig:dimension_gaussian}. We use an autonomous system $x_{t+1} = A^\star x_{t} + w_{t}$, where $A \in \mathbb{R}^{n_x\times n_x}$ has varying $n_x$. Disturbances $w_t$ are sampled from a truncated Gaussian distribution  and uniform distribution  with $w_{\max}=2$. {Surprisingly, the scaling of SME with respect to the dimension of the system is not significantly worse than that of LSE in the simulation. This suggests that the  convergence rate in \Cref{cor: uniform dist est. bdd} can potentially be  improved in terms of the dimension dependence, which is left for future investigation.} 

\begin{figure}[ht]
    \centering
     \includegraphics[width=0.38\textwidth, trim=0cm 6cm 0cm 7.2cm,clip]{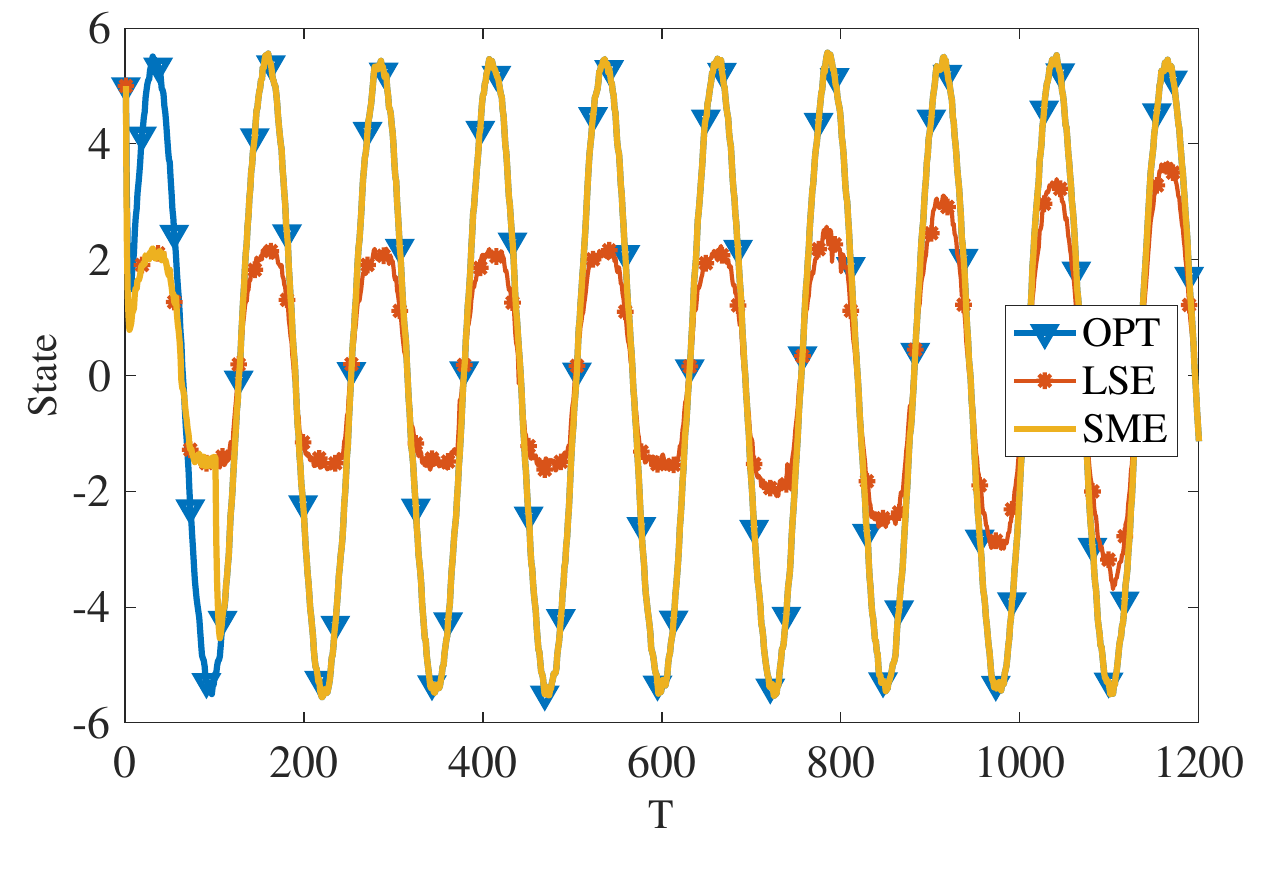}
    \vspace{-0.15cm}
    \caption{Linear quadratic tracking  of robust adaptive MPC based on SME, LSE's confidence regions, and the accurate model (OPT).}
    \label{fig:ra-mpc}
\end{figure}

\textbf{Application to robust adaptive MPC. } We provide an example of the quantitative impact of using SME for adaptive robust MPC in \Cref{fig:ra-mpc}. We consider the task of constrained linear quadratic tracking problem as in \cite{li2023non}. The model uncertainty set is estimated online with SME and LSE's 90 \% confidence region. Control actions are computed using the  tube-based robust MPC \cite{rawlings2017model,mayne2005robust} with the uncertainty sets. We also plot the  optimal MPC controller with accurate model information. Thanks to the fast convergence of  SME, the tracking performance of the tube-based robust MPC with SME estimation quickly coincides with OPT, while the same controller based on LSE's confidence region estimation converges more slowly.

\vspace{-10pt}

\section{Concluding Remarks}
\vspace{-5pt}

This work provides the first convergence rates for SME in linear dynamical systems with bounded disturbances and discusses variants of SME with unknown bound on $w_t$. Numerical experiments demonstrate SME's promising performance under bounded disturbances.

\vspace{-1pt}
Regarding future directions, this work only considers box constraints on $w_t$, so it is worth extending the analysis to more general constraints. This paper only measures the size of the uncertainty sets by their diameters. We leave for future work to consider other metrics, such as volume. Further, our bounds suffer poor dependence on the system dimension, which is not reflected in simulations. Hence, it is important to refine the bounds and discuss the fundamental limits. Another exciting direction is to speed up the computation of SME since the current computation complexity increases linearly with the sample size. The convergence rate of such algorithms is an important \textit{open question}. Other interesting directions include the extensions of the SME analysis to nonlinear systems, where recent nonlinear system identification literature \cite{sattar2022finite,foster2020learning} may provide insights; and analyzing SME in the presence of other uncertainties, e.g. measurement noises \cite{sarkar2019near}.

\vspace{-1pt}
SME is a valid estimation for bounded {non-stochastic} disturbances \cite{fogel1982value,milanese2013bounding, lauricella2020set,livstone1996asymptotic}. Thus, a fruitful direction is to study SME's convergence rates under non-stochastic $w_t$. Another  method for uncertainty set estimation is the credible regions of Bayesian approaches, e.g. Thompson sampling for linear systems \cite{kargin2022thompson,abeille2017thompson} and Gaussian processes for nonlinear systems \cite{fisac2018general}. A future direction is to study the convergence rates of credible regions.

\section*{Impact Statement}
This paper presents work whose goal is to advance the field of 
Machine Learning. There are many potential societal consequences 
of our work, none of which we feel must be specifically highlighted here.

\bibliography{citation4setmember}

\bibliographystyle{icml2024}


\newpage
\appendix
\onecolumn



\section*{Roadmap for the appendices}
\begin{itemize}
    \item \Cref{subsec:notation} introduces additional notation used throughout the Appendix.
    \item \Cref{subsec:related_work} provides more  literature review on LSE and SM,  and  a more detailed discussion on the technical contributions of this paper.
    \item \Cref{appendix:assumption} provides more discussions on  examples that satisfy Assumptions \ref{ass: BMSB, bounded xt} and \ref{ass: tight bound on wt}.
    \item \Cref{appendix:proof_main} presents the  proof of \Cref{thm: estimation err bdd}. In particular, we provide  helper lemmas in \Cref{sec:D1} and prove \Cref{lem: bound PE c}, \Cref{lem: bound PE and gamma>delta/2} in \Cref{sec:D2} and \Cref{sec:D3} respectively. A more precise upper bound for Theorem \ref{thm: estimation err bdd} (without the $\tilde O(\cdot)$ notation) is provided in Appendix \ref{append: precise upper bdd}.
    \item \Cref{sec:cor1} presents a  proof of \Cref{cor: uniform dist est. bdd}
     \item \Cref{sec:cor2} provides a proof of \Cref{cor: estimation err bdd for B=0}.
    \item \Cref{sec:thrm2} presents a  proof of \Cref{thm: loose bound}.
   \item Appendix \ref{append: unknown wmax} provides proofs of Theorem 3, Corollary 3, and Theorem 4.

    \item \Cref{append: simulation} provides details of the simulation.

\end{itemize}


\section{Additional notations}
\label{subsec:notation}
	Let $\Sb_n(0,1)$ denote the unit sphere in $\R^n$ in $l_2$ norm, i.e., $\Sb_n(0,1)=\{x\in \R^n: \|x\|_2=1\}$. Let $\Sb_{n\times m}(0,1)$ denote the unit sphere in $\R^{n\times m}$ with respect to the Frobenius norm, i.e., $\Sb_{n\times m}(0,1)=\{M\in \R^{n\times m}: \|M\|_F=1\}$. Let $\bar B_n(0,1)$ denote the closed unit ball  in $\R^n$ in $l_2$ norm, i.e., $\bar B_n(0,1)=\{x\in \R^n: \|x\|_2\leq1\}$. Let  $\bar B_{n\times m}(0,1)$ denote the closed unit ball in $\R^{n\times m}$ in Frobenius norm, i.e., $\bar B_{n\times m}(0,1)=\{M\in \R^{n\times m}: \|M\|_F\leq 1\}$. For a matrix $M\in \R^{n\times m}$, $\text{vec}(M)$ is the vectorization of $M$. Moreover, we define the inverse mapping of $\text{vec}(\cdot)$ as $\text{mat}(\cdot)$, i.e.,  for a vector $d\in \R^{nm}$, $\text{mat}(d)\in \R^{n\times m}$. Consider a $\sigma$-algebra $\F$ and a random variable $X$, we write $X\in \F$ if $X $ is measurable with respect to $\F$, i.e., for all Borel measurable sets $B\subseteq \R$, we have $X^{-1}(B)\in \F$. We can similarly define $\F$-measurable random matrices and random vectors. Further, consider a  polyhedral $\D=\{x: Ax\leq b\}$, we write $\D\in \F$ if matrix $A$ and vector $b$ are measurable with respect to $\F$. Consider two symmetric matrices $A,B \in \R^{n\times n}$, we write $A\succeq B$ if $A-B$ is a positive definite matrix. We define $\min \emptyset =+\infty$. For a set $\Ec$, let $\one_{\Ec}$ denote the indicator function on $\Ec.$ For a vector $x\in\mathbb{R}^n$, we use $x^j$ to denote the $j$th coordinate of $x$. Throughout the paper, we use
$\text{TrunGauss}(0,\sigma_w, [-w_{\max}, w_{\max}])$ to refer to the truncated Gaussian distribution generated by Gaussian distribution with zero mean and $\sigma_w^2$ variance with truncated range $[-w_{\max}, w_{\max}]$. The same applies to multi-variate truncated Gaussian distributions.


\section{More discussions on  least square and set membership}
\label{subsec:related_work}
System identification studies the problem of estimating the parameters of an unknown dynamical systems from trajectory data. There are two main classes of estimation methods: point estimator such as least square estimation (LSE), and set estimator such as set membership estimation (SME). In the following, we provide more discussions and literature review on  LSE and SME. We will also discuss the major technical novelties of this work.

\subsection{Least square estimation}
For linear dynamical systems $x_{t+1}=A^*x_t +B^*u_t+w_t=\theta^* z_t+w_t$, given a trajectory of data $\{x_t, u_t\}_{t\geq 0}$, least square estimation generates a point estimator that minimizes the following quadratic error \cite{van2012subspace, ljung1998system}:
\begin{align*}
    \hat \theta_{\text{LSE}}=\min_{\hat \theta}\sum_{t=0}^{T-1}\|x_{t+1}-\hat \theta z_t\|_2^2.
\end{align*}

Least-square estimation is widely used and its convergence (rate) guarantees have been investigated for a long time. 
In particular, non-asymptotic  convergence rate guarantees of LSE has become increasingly important as these guarantees are the foundations for non-asymptotic performance analysis of learning-based/adaptive control algorithms. Earlier non-asymptotic analysis of LSE  focused on the simpler regression model $y_t=\theta^* x_t +w_t$, where $x_t$ and $y_t$ are independent \cite{campi2002finite, vidyasagar2006learning, hsu2012random}.  

Recently, there is one major breakthrough  in \cite{simchowitz2018learning} that provides LSE's convergence rate analysis for linear dynamical system $x_{t+1}=\theta^* z_t +w_t$, where $x_{t+1}$ and $z_t=[x_t^\top, u_t^\top]^{\top}$ are correlated. More specifically, \cite{simchowitz2018learning} establishes a fundamental property, \textit{block-martingale small-ball (BMSB)}, to analyze LSE under correlated data. BMSB enables a long list of subsequent literature on LSE's non-asymptotic analysis for different types of dynamical systems, e.g., \cite{oymak2019non,dean2019sample,zheng2020non,rantzer2018concentration,faradonbeh2018finite, wagenmaker2020active, tsiamis2022statistical, zhao2022adaptive,li2021safe}.

Though LSE is a point estimator, one can establish confidence region of LSE based on proper statistical assumptions on $w_t$. The pioneer works on the confidence region of LSE for linear dynamical systems are \cite{abbasi2011online,abbasi2011regret}, which construct ellipsoid confidence regions for LSE. Moreover, the non-asymptotic bounds on estimation errors established in \cite{simchowitz2018learning,dean2019safely} can also be viewed as confidence bounds. Further, the estimation error $\tilde O(\frac{\sqrt{n_x+n_z}}{\sqrt T})$ has been shown to match the fundamental lower bound for any estimation methods for unbounded disturbances in \cite{simchowitz2018learning}. However, these confidence bounds all rely on statistical inequalities, which may result in loose constant factors despite an optimal convergence rate. When applying these confidence bounds to robust control, where the controller is required to satisfy certain stability and constraint satisfaction properties for every possible system in the confidence region, a loose constant factor will result in a larger confidence region and a more conservative control design. Finally, in robust control and many practical applications, the disturbances are usually bounded, and it will be interesting to see how the knowledge of the boundedness will improve the uncertainty set estimation.

On a side note, this paper is also related with the ambiguity set estimation for the transition probabilities in robust Markov decision processes \cite{petrik2019beyond}. There are  attempts on improving the ambiguity set estimation based on LSE for less conservative robust MDP \cite{petrik2019beyond}.

\subsection{Set membership}
\label{subsec:sm}

Set membership is commonly used in robust control for uncertainty set estimation \cite{milanese1991optimal,adetola2011robust, tanaskovic2013adaptive,bujarbaruah2020adaptive,zhang2021trajectory, parsi2020robust, parsi2020active, sasfi2022robust}. 
There is a long history of research on SME for both deterministic disturbances, such as \cite{bai1995membership,fogel1982value, kitamura2003convergence,milanese2013bounding, lauricella2020set,livstone1996asymptotic}, and 
 stochastic disturbances, such as \cite{bai1998convergence, bai1995membership, kitamura2005size,akccay2004size,lu2019robust}. For the stochastic disturbances, both convergence and convergence rate analysis have been investigated under the persistent excitation (PE) condition. However, the existing convergence rates are only established for simpler regression problems, $y_t=\theta^*x_t +w_t$, where $y_t$ and $x_t$ are independent \cite{akccay2004size,bai1995membership,bai1998convergence,kitamura2005size}. 

 Recently, \cite{lu2019robust} provided an initial attempt to establish the convergence guarantee of SME for linear dynamical systems $x_{t+1}=\theta^* z_t+w_t$ for correlated data $x_{t+1}$ and $z_t$. However, \cite{lu2019robust} assumes that PE holds deterministically, and designs a special control design based on constrained optimization to satisfy PE deterministically. Therefore, the convergence for general control design and the convergence rate analysis remain open questions for correlated data arising from dynamical systems.

 In this paper, we establish the convergence rate guarantees of SME on linear dynamical systems under the BMSB conditions in \cite{simchowitz2018learning}. Compared with \cite{lu2019robust}, BMSB condition can be satisfied by adding an i.i.d. random noise to a general class of control designs \cite{li2021safe}. 

\nbf{Technically}, one major challenge of SME analysis compared with the LSE analysis  is that the diameter of the membership set does not have an explicit formula, which is in stark contrast with LSE, where the point estimator is the solution to a quadratic program and has explicit form. A common trick to address this issue in the analysis of SME is to connect the diameter bound with the values of disturbances subsequences $\{w_{s_k}\}_{k\geq 0}$: it can be generally shown that a large diameter indicates that a long subsequence of disturbances are far away from the boundary of $\W$. However, existing construction methods of $\{w_{s_k}\}_{k\geq 0}$ will cause the time indices $\{s_k\}_{k\geq 0}$ to \textit{correlate} with the realization of the sequences $\{x_t, u_t, w_t\}_{t\geq 0}$ \cite{akccay2004size,lu2019robust,bai1995membership}.\footnote{In \cite{lu2019robust}, the correlation between $\{s_k\}_{k\geq 0}$ and $\{x_t, u_t,w_t\}_{t\geq 0}$ is via the PE condition, but \cite{lu2019robust} assume  deterministic PE to avoid this correlation issue.} Consequently, in the correlated-data scenario and when PE does not hold deterministically, under the existing construction methods in \cite{akccay2004size,lu2019robust,bai1995membership}, the probability of $\{w_{s_k}\}_{k\geq 0}$ with correlated time indices \textit{cannot} be  bounded by the probability of the \textit{independent} sequence  $\{w_t\}_{t\geq 0}$. One major \nbf{technical contribution} of this paper is to provide a novel construction of  $\{w_{s_k}\}_{k\geq 0}$ based on a sequence of stopping times and establish conditional independence properties despite correlated data and stochastic PE condition (BMSB). More details can be found in Lemma \ref{lem: construct Gik and Lik} and the proof or Lemma \ref{lem: bound PE and gamma>delta/2}.

Though we only consider box constraints for $w_t$, it is worth mentioning that SME can be applied to much more general forms of disturbances. For example, a common alternative is the ellipsoidal-bounded disturbance where $\mathcal{W}:=\{ w\in \mathcal{R}^{n_x}: w^\top P w \leq 1 \}$ with positive definite $P \in \mathcal{R}^{n_x \times n_x}$ \cite{bai1995membership,van2023informativity, eising2023sampling,liu2016ellipsoidal} and polytopic-bounded disturbance $\mathcal{W}:=\{ w\in \mathcal{R}^{n_x}: Gw \leq h \}$ for positive definite $G\in\mathcal{R}^{n_x \times n_x}$ and $h\in\mathcal{R}^{n_x}$ \cite{fogel1982value,lu2019robust,lu2023robust}. There are also SME literature assuming bounded energy of the disturbance sequences \cite{bai1995membership}. It is an interesting future direction to extend the analysis in this paper to more general disturbance constraints.



Further, exact SME involves the intersection of an increasing number of sets, thus causing the computation complexity increases with time $t$, which can become prohibitive when $t$ is large. There are many methods trying to reduce the computation complexity by approximating the membership sets (see e.g., \cite{livstone1996asymptotic,lu2019robust}, etc.). It is an exciting future direction to study the diameter bounds of the approximated SME methods.

Lastly, it is worthing mentioning that SME can also be applied to the uncertainty set estimation in perception-based control \cite{gao2024closure,tang2023uncertainty}.






\section{More discussions on Assumptions \ref{ass: BMSB, bounded xt} and \ref{ass: tight bound on wt}}
\label{appendix:assumption}

\subsection{More discussions on Assumption \ref{ass: BMSB, bounded xt}}
The BMSB condition has been widely used in learning-based control. It has been shown that BMSB can be satisfied in many scenarios. For example, \cite{simchowitz2018learning,tu2019sample} showed that linear systems with i.i.d. perturbed linear control policies, i.e., $x_{t+1}=Ax_t+B( Kx_t+\eta_t)+w_t$,\footnote{Though we only describe a static linear policy $u_t=Kx_t$ here, the results in \cite{simchowitz2018learning,tu2019sample,dean2019safely} hold for dynamic linear policies.} satisfy BMSB if the disturbances $w_t$ and $\eta_t$  are i.i.d. and follow Gaussian distributions with positive definite covariance matrices. Later, \cite{dean2019safely} showed that $x_{t+1}=Ax_t+B( Kx_t+\eta_t)+w_t$ can still satisfy BMSB even for non-Gaussian distributions of $w_t,\eta_t$, as long as $w_t$ and $\eta_t$ have independent coordinates and finite fourth moments. Recently, \cite{li2021safe} extended the results to linear systems with nonlinear policies, i.e., $x_{t+1}=Ax_t+B(\pi_t(x_t)+\eta_t)+w_t$, and showed that BMSB still holds as long as the nonlinear policies $\pi_t$ generate bounded trajectories of states and control inputs, and $w_t, \eta_t$ are bounded and follow distributions with certain anti-concentrated properties (a special case is positive definite covariance matrix).


\subsection{ More discussions on Assumption \ref{ass: tight bound on wt}}\label{append: qw example}
In this subsection, we provide two example distributions, truncated Gaussian and uniform distributions, and discuss their corresponding $q_w(\epsilon)$ functions. It will be shown that for both distributions below, $q_w(\epsilon)=O(\epsilon)$.

	\begin{lemma}[Example of uniform distribution]\label{lem: uniform dist on Wcal qw(epsilon)}
		Consider $w_t$ that follows a uniform distribution on $[-w_{\max}, w_{\max}]^{n_x}$. Then, $q_w(\epsilon)=\frac{\epsilon}{2w_{\max}}$.
	\end{lemma}
	\begin{proof}
		Since $\textup{Unif}(\W)$ is symmetric, we only need to consider one direction $j=1$. 
		\begin{align*}
			\Pb(w^j+w_{\max}\leq \epsilon)&=\int_{w^1+w_{\max}\leq \epsilon} \int_{w^2, \dots, w^{n_x}\in [-w_{\max}, w_{\max}]} \frac{1}{(2w_{\max})^{n_x}} \one_{(w\in \W)}\,\mathrm d w\\
			&=		\int_{w^1\leq \epsilon-w_{\max}}\frac{1}{2w_{\max}} \one_{(w\in \W)}\mathrm d w^1= \frac{\epsilon}{2w_{\max}}
		\end{align*}
		Similarly,
$				\Pb(w_{\max}-w^1\leq \epsilon) = 	\int_{w^1\geq w_{\max}-\epsilon}\frac{1}{2w_{\max}} \one_{(w\in \W)}d w^1= \frac{\epsilon}{2w_{\max}}
$.
	\end{proof}

 
	\begin{lemma}[Example of truncated Gaussian distribution]\label{lem: truncated Gaussian dist on Wcal qw(epsilon)}
		Consider $w_t$ follows a truncated Gaussian distribution on $[-w_{\max}, w_{\max}]^{n_x}$ generated by a Gaussian distribution with zero mean and $\sigma_w I_{n_x}$ covariance matrix. Then, $q_w(\epsilon)=\frac{1}{\min(\sqrt{2\pi}\sigma_w, 2w_{\max})} \exp(\frac{-w_{\max}^2}{2\sigma_w^2})\epsilon$. 
	\end{lemma}
	\begin{proof}
		Since this distribution is symmetric and each coordinate is independent, we only need to consider one direction $j$. Let $X$ denote a Gaussian distribution with zero mean and $\sigma_w^2$ variance. By the definition of truncated Gaussian distributions, we have
		\begin{align*}
			\Pb(w^j+w_{\max}\leq \epsilon)&=\frac{\Pb(-w_{\max}\leq  X \leq -w_{\max}+\epsilon)}{\Pb(-w_{\max}\leq X \leq w_{\max})}
		\end{align*}
  Notice that $X/\sigma_w$ follows the standard Gaussian distribution, so we can obtain the following bounds.
  \begin{align*}
      \Pb(-w_{\max}\leq  X \leq -w_{\max}+\epsilon)& =\int_{-w_{\max}/\sigma_w}^{(-w_{\max}+\epsilon)/\sigma_w} \frac{1}{\sqrt{2\pi}}\exp(-\frac{z^2}{2})\,\mathrm d z\\
      & \geq \frac{1}{\sqrt{2\pi}}\exp(-w_{\max}^2/(2\sigma_w^2) )\frac{\epsilon}{\sigma_w}
  \end{align*}
and
\begin{align*}
      \Pb(-w_{\max}\leq  X \leq w_{\max})& =\int_{-w_{\max}/\sigma_w}^{w_{\max}/\sigma_w} \frac{1}{\sqrt{2\pi}}\exp(-\frac{z^2}{2})\,\mathrm d z\\
      & \leq \min(1, \frac{1}{\sqrt{2\pi}}\frac{2w_{\max}}{\sigma_w})
  \end{align*}
  Therefore, we obtain
  \begin{align*}
			\Pb(w^j+w_{\max}\leq \epsilon)&=\frac{\Pb(-w_{\max}\leq  X \leq -w_{\max}+\epsilon)}{\Pb(-w_{\max}\leq X \leq w_{\max})}\\
   &\geq \max(\frac{1}{\sqrt{2\pi}}\exp(-w_{\max}^2/\sigma_w^2) \frac{\epsilon}{\sigma_w}, \frac{\epsilon}{2w_{\max}} \exp(\frac{-w_{\max}^2}{2\sigma_w^2}))\\
   & = \frac{1}{\min(\sqrt{2\pi}\sigma_w, 2w_{\max})} \exp(\frac{-w_{\max}^2}{2\sigma_w^2})\epsilon
		\end{align*}
Finally, $				\Pb(w_{\max}-w^1\leq \epsilon)$ can be bounded similarly.
	\end{proof}

\begin{lemma}[Example of uniform distribution on the boundary of $\W$ (a generalization of Rademacher distribution)] Consider $w_t$ follows a uniform distribution on $\{w: \|w\|_\infty =w_{\max}\}$. Then $q_w(\epsilon)=\frac{1}{2n_x}$.
    
\end{lemma}
\begin{proof}
    Since the hyper-cube $\{w: \|w\|_\infty =w_{\max}\}$ has $2n_x$ facets, the probability on each facet is $\frac{1}{2n_x}$. Therefore, $\Pb(w^j \leq \epsilon -w_{\max}) \geq \Pb(w^j=-w_{\max})=\frac{1}{2n_x}$ for all $j$. The same applies to $\Pb(w^j \geq -\epsilon +w_{\max})$.
    \end{proof}

\section{Proof of Theorem \ref{thm: estimation err bdd}}\label{appendix:proof_main}



	The section provides more details for the proof of Theorem \ref{thm: estimation err bdd}. In particular,  we first provide technical lemmas for set discretization, then prove Lemma \ref{lem: bound PE c} and Lemma \ref{lem: bound PE and gamma>delta/2} respectively. The proof of Theorem \ref{thm: estimation err bdd} follows naturally by combining Lemma \ref{lem: bound PE c} and Lemma \ref{lem: bound PE and gamma>delta/2}.







	\subsection{Technical lemmas: set discretization}
 \label{sec:D1}
 This subsection provide useful technical lemmas for the proofs of Lemma \ref{lem: bound PE c} and Lemma \ref{lem: bound PE and gamma>delta/2}. The results are based on a finite-ball covering result that is classical in the literature \cite{rogers1963covering} \cite{verger2005covering}. 
	

	\begin{theorem}[Theorem 1.1 and 1.2 in \cite{verger2005covering} and Theorem 2 in \cite{rogers1963covering} (revised to match the setting of this paper)]\label{thm: ball cover}
		Consider a closed ball $\bar \B_n(0,1)=\{x\in \R^n: \|x\|_2\leq 1\}$ in $l_2$ norm. Considering covering this ball $\bar \B_n(0,1)$ with smaller closed balls $\bar \B_n(z,\epsilon)$ for $z\in \R^n$. Let $v_{\epsilon, n}$ denote the minimal number of smaller balls needed to cover $\bar \B_n(0,1)$. For $n\geq 1$ and $0<\epsilon<1/2$, we have
		\begin{align*}
			 v_{\epsilon, n} & \leq 544 n^{2.5} \log(n/\epsilon) (\frac{1}{\epsilon})^n
		\end{align*}
	\end{theorem}
 \begin{proof}
     Theorem 1.1 and 1.2 in \cite{verger2005covering} and Theorem 2 in \cite{rogers1963covering} discuss the upper bounds of $v_{\epsilon,n}$ in several different cases. These upper bounds in these different cases are unified by the upper bound in the theorem above by algebraic manipulations. 
 \end{proof}
		
		
		

	

 We apply Theorem \ref{thm: ball cover} to obtain the number of covering balls in the two settings below. These two settings will be considered in the proofs of Lemma \ref{lem: bound PE c} and \ref{lem: bound PE and gamma>delta/2} respectively.
	\begin{corollary}\label{cor: mesh on lambda}
		There exists a finite set $\mathcal M'=\{\lambda_1, \dots, \lambda_{v_\lambda}\}\subseteq \Sb_{n_z}(0,1)$ such that for any $\lambda \in \R^{ n_z}$ with $\|\lambda\|_2=1$, there exists $\lambda_i\in \mathcal M'$ such that $\|\lambda-\lambda_i\|_2 \leq 2\epsilon_\lambda$. 
		
		In the following, we consider $\epsilon_\lambda = \sigma_z^2 p_z^2/(64 b_z^2)=1/a_2$. Notice that  $\epsilon_\lambda <1/2$. 
		Accordingly, 
  \begin{equation}
  \label{eq:v_lambda}
  v_{\lambda}\leq 544 n_z^{2.5} \log(a_2 n_z)a_2^{n_z}.
  \end{equation}
  
	\end{corollary}
 \begin{proof}
     $\epsilon_{\lambda}\leq 1/64<1/2$ because $p_z\leq 1$ and $\sigma_z\leq b_z$ by the definitions of BMSB and $b_z$. Then, the bound on $v_{\lambda}$ follows from Theorem \ref{thm: ball cover}.
 \end{proof}

	\begin{lemma}\label{lem: mesh on unit sphere of matrices}
		There exists a finite set $\mathcal M=\{\gamma_1, \dots, \gamma_{v_\gamma}\}\subseteq \Sb_{n_x\times n_z}(0,1)$ such that for any $\gamma \in \R^{n_x\times n_z}$ and $\|\gamma\|_F=1$, there exists $\gamma_i\in \mathcal M$ such that $\|\gamma-\gamma_i\|_F \leq 2\epsilon_\gamma$. Consider $\epsilon_\gamma = \frac{a_1}{4b_z \sqrt n_x}=1/a_4$. Notice that $\epsilon_\gamma<1/2$. Accordingly,  $$v_\gamma\leq 544 n_x^{2.5}n_z^{2.5} \log(a_4 n_x n_z)a_4^{n_zn_x}.$$
		
	\end{lemma}
	\begin{proof}
 The proof is basically by mapping  the matrices to  vectors based on matrix vectorization, then mapping the vectors back to matrices. These two mappings are isomorphism.
 
		Specifically, consider a closed unit ball in $\R^{n_xn_z}$. There exist $v_{\epsilon, n_xn_z}$ smaller closed balls to cover it, denoted by $\B_1, \dots, \B_{v_{\epsilon, n_xn_z}}$. Consider the non-empty sets from $\B_1\cap \Sb_{n_x n_z}(0,1), \dots, \B_{v_{\epsilon, n_x n_z}}\cap \Sb_{n_x n_z}(0,1)$. For any $1\leq i \leq v_{\epsilon, n_x n_z}$, if $\B_i \cap \Sb_{n_x n_z}(0,1)\not = \emptyset$, select a point $\text{vec}(\gamma)\in \B_i \cap \Sb_{n_x n_z}(0,1) $. Notice that $\|\text{vec}(\gamma)\|_2=1$. In this way, we construct a finite sequence $\{\text{vec}(\gamma_1), \dots, \text{vec}(\gamma_{v_{\gamma}})\}$ where $v_\gamma\leq v_{\epsilon_\gamma, n_xn_z}$.\footnote{Here, without loss of generality, we consider $\B_1\cap \Sb_{n_xn_z}(0,1), \dots, \B_{v_\gamma}\cap \Sb_{n_xn_z}(0,1)$ are not empty.} 
		
		For any $\gamma \in \R^{n_x\times n_z}$, we have $\text{vec}(\gamma)\in \R^{n_xn_z}$ and $\|\text{vec}(\gamma)\|_2=1$. Hence, there exists $1\leq i \leq v_\gamma$ such that $\text{vec}(\gamma)\in \B_i \cap \Sb_{n_xn_z}(0,1)$. Hence, $\|
		\text{vec}(\gamma)-\text{vec}(\gamma_i)\|_2\leq 2\epsilon_\gamma$. Moreover, $\|\gamma_i\|_F=\|\text{vec}(\gamma_i)\|_2=1$. Therefore, $\|\gamma_i-\gamma\|_F\leq 2\epsilon_\gamma$. So the set $\mathcal M=\{\gamma_1, \dots, \gamma_{v_\gamma}\}$ satisfies our requirement.
  	\end{proof}

		


\subsection{Proof of Lemma \ref{lem: bound PE c}}
\label{sec:D2}
Essentially, Lemma \ref{lem: bound PE c} shows that PE holds with high probability under the BMSB condition. This result has been  established in Proposition 2.5 in \cite{simchowitz2018learning}, though in a different form. The rest of this subsection will prove the PE condition needed in this paper based on   Proposition 2.5 in \cite{simchowitz2018learning}.

Firstly, we review Proposition 2.5 in \cite{simchowitz2018learning} for the convenience of the reader.
\begin{theorem}[Proposition 2.5 in \cite{simchowitz2018learning} when $k=1$]\label{prop: prop 2.5 in LWM paper}
Let $\{Z_t\}_{t\geq 1}$ be an $\{\F^Z_t\}_{t\geq 1}$-adapted random process taking values in $\R$. $Z_0$ is given. If  $\{Z_t\}_{t\geq 0}$  is $(1,v,p)$-BMSB,  then 
	$$\Pb(\sum_{t=1}^T Z_t^2 \leq v^2 p^2T/8)\leq \exp(-T p^2/8)$$
\end{theorem}

Next, we prove the PE in one segment of data sequence.

\begin{lemma}[Probability of PE in one segment]\label{lem: PE in one segment}
	For any $m\geq 1$, for any $k\geq 0$, we have
	$$\Pb(\sum_{t=km+1}^{km+m} z_t z_t^\top \succ (\sigma_z^2 p_z^2m/16)  I_{n_z}\mid \F_{km}) \geq 1-v_{\lambda} \exp(-m p_z^2/8))$$
\end{lemma}
\begin{proof}
Consider $\M'=\{\lambda_1, \dots, \lambda_{v_{\lambda}}\}$ defined in 
 Corollary \ref{cor: mesh on lambda}. For any $\lambda_i\in\M'$, $\lambda_i^\top z_t$ satisfies the $(1,\sigma_z, p_z)$-BMSB condition. Therefore, by Theorem \ref{prop: prop 2.5 in LWM paper},  we have
 $$\Pb(\sum_{t=1}^T \lambda_i^\top z_t z_t^\top \lambda_i \leq \sigma_z^2 p_z^2T/8)\leq \exp(-T p_z^2/8).$$
 Notice that the horizon length $T$ is arbitrary and the starting stage $t=1$ can also be different because we consider a time-invariant dynamical system in this paper. Therefore, for any $m\geq 1, k\geq 0$, for any $\lambda_i\in\M'$, we  have
 $$\Pb(\sum_{i=1}^m \lambda_i^\top z_{km+i} z_{km+i}^\top \lambda_i \leq \sigma_{z}^2 p_{z}^2m/8 \mid \F_{km})\leq \exp(-m p_{z}^2/8),$$
 where we condition on $\F_{km}$ to make sure $z_{km}$ is known under $\F_{km}$, which is required by Theorem \ref{prop: prop 2.5 in LWM paper}.
 
 For arbitrary $\lambda$ such that $\|\lambda\|_2=1$,   there exists $\lambda_i\in \M'$ such that $\|\lambda-\lambda'\|_2 \leq 2\epsilon_\lambda$. 
 Therefore, we can bound $\sum_{t=km+1}^{km+m} \lambda^\top z_t z_t^\top \lambda$ by $\sum_{t=km+1}^{km+m} \lambda_i^\top z_t z_t^\top \lambda_i$. 
	\begin{align*}
		\sum_{t=km+1}^{km+m} \lambda^\top z_t z_t^\top \lambda&= \sum_{t=km+1}^{km+m} \lambda_i^\top z_t z_t^\top \lambda_i + \sum_{t=km+1}^{km+m} (\lambda+\lambda_i)^\top z_t z_t^\top (\lambda-\lambda_i ) \\
		& \geq \sum_{t=km+1}^{km+m} \lambda_i^\top z_t z_t^\top \lambda_i  -  \sum_{t=km+1}^{km+m}\|\lambda+\lambda_i\|_2 \|z_t\|_2^2\|\lambda_i-\lambda\|_2\\
		& \overset{(a)}{\geq} \sum_{t=km+1}^{km+m} \lambda_i^\top z_t z_t^\top \lambda_i  -  \sum_{t=km+1}^{km+m} 4 b_z^2 \epsilon_\lambda\\
		& =\sum_{t=km+1}^{km+m} \lambda_i^\top z_t z_t^\top \lambda_i  - 4 b_z^2 \epsilon_\lambda m \overset{(b)}{\geq} \sum_{t=km+1}^{km+m} \lambda_i^\top z_t z_t^\top \lambda_i  - \sigma_z^2 p_z^2m/16
	\end{align*}
where $(a)$ is by Assumption \ref{ass: BMSB, bounded xt}, $\|\lambda-\lambda_i\|_2 \leq 2\epsilon_\lambda$, and $\|\lambda\|_2=\|\lambda_i\|_2=1$; and $(b)$ is by choosing $\epsilon_\lambda\leq\sigma_z^2 p_z^2/(64 b_z^2)$.

Therefore, by the definition of positive definiteness and the inequalities above, we can complete the proof by the following.
\begin{align*}
   \Pb(\sum_{t=km+1}^{km+m} z_t z_t^\top \succ (\sigma_z^2 p_z^2m/16) I_{n_z} \mid \F_{km})&= \Pb( \forall \, \|\lambda\|_2=1, \ \sum_{t=km+1}^{km+m} \lambda^\top z_t z_t^\top \lambda > \sigma_z^2 p_z^2m/16\mid \F_{km})\\
	& \geq \Pb( \forall 1\leq i \leq v_{\lambda}, \ \sum_{t=km+1}^{km+m} \lambda_i^\top z_t z_t^\top \lambda_i > \sigma_z^2 p_z^2m/8\mid \F_{km})\\
 & \geq 1-\sum_{i=1}^{v_{\lambda}}\Pb(\sum_{t=km+1}^{km+m} \lambda_i^\top z_t z_t^\top \lambda_i \leq \sigma_z^2 p_z^2m/8\mid \F_{km})\\
 & \geq 1-v_{\lambda} \exp(-m p_z^2/8)),
\end{align*}
which completes the proof.
\end{proof}

Now, we are ready for the proof of Lemma \ref{lem: bound PE c}.
\begin{proof}[\textbf{Proof of Lemma \ref{lem: bound PE c}}]
	Recall that $	\Ec_2=	\{\frac{1}{m}\sum_{{s=1}}^{{m}} {z_{km+s} z_{km+s}^\top} \succeq  a_1^2 I_{n_z}, \ \forall\, 0\leq k \leq \ceil{{T}/m}-1 \}$, where $a_1=\sigma_z p_z/4$. Hence
	$$ \Ec_2=\bigcap_{k=0}^{T/m-1}\{\sum_{t=km+1}^{km+m} z_t z_t^\top \succ (\sigma_z^2 p_z^2m/16) I_{n_z}\}.$$
 Therefore, 
	\begin{align*}
\Pb(\Ec_2)
		& \geq  1- \sum_{k=0}^{T/m-1} \Pb(\sum_{t=km+1}^{km+m} z_t z_t^\top \preceq (\sigma_z^2 p_z^2m/16) I_{n_z})\\
  & \geq 1- \frac{T}{m} v_{\lambda}\exp(-m p_z^2/8)\\
  & = 1- \frac{T}{m} (544 n_z^{2.5} \log(a_2 n_z)a_2^{n_z}) \exp(-m p_z^2/8),
	\end{align*}
 where we use Lemma \ref{lem: PE in one segment} and the fact that if $\Pb(\sum_{t=km+1}^{km+m} z_t z_t^\top \preceq (\sigma_z^2 p_z^2m/16)  I_{n_z}\mid \F_{km}) \leq v_{\lambda} \exp(-m p_z^2/8))$, then $\Pb(\sum_{t=km+1}^{km+m} z_t z_t^\top \preceq (\sigma_z^2 p_z^2m/16)  I_{n_z}) \leq v_{\lambda} \exp(-m p_z^2/8))$.
\end{proof}

\subsection{Proof of Lemma \ref{lem: bound PE and gamma>delta/2}}
\label{sec:D3}

This proof takes four major steps:
\begin{enumerate}
	\item[(i)] Define $b_{i,t}, j_{i,t}, L_{i,k}$.
	\item[(ii)] Provide a formal definition of $\Ec_{1,k}$ based on $b_{i,t}, j_{i,t}, L_{i,k}$ and prove a formal version of Lemma \ref{lem: discretize E1 E2}.
	\item[(iii)] Prove Lemma \ref{lem: construct Gik and Lik}.
	\item[(iv)] Prove Lemma \ref{lem: bound PE and gamma>delta/2} by the formal version of Lemma \ref{lem: discretize E1 E2} and Lemma \ref{lem: construct Gik and Lik}.
\end{enumerate}
It is worth mentioning that the formal definition of $\Ec_{1,k}$ is slightly different from the definition in Lemma \ref{lem: discretize E1 E2}, but we still have $\Pb(\Ec_1\cap \Ec_2)\leq \sum_{i=1}^{v_{\gamma}}\Pb(\Ec_{1,k}\cap \Ec_2)$, which is the key property that will be used in the proof of Lemma \ref{lem: bound PE and gamma>delta/2}. 

\subsubsection{Step (i): Definitions of $b_{i,t}, j_{i,t}, L_{i,k}$.}\label{append: define bi ji Li}

Recall the  discretization of $\Sb_{n_x\times n_z}(0,1)$ in Lemma \ref{lem: mesh on unit sphere of matrices}, which generates the set $\M=\{\gamma_1, \dots, \gamma_{v_{\gamma}}\}$. We are going to define $b_{i,t}, j_{i,t}, L_{i,k}$ for $\gamma_i \in \M$ for each $1\leq i \leq v_{\gamma}$. Notice that $\M$ is a deterministic set of matrices.

\begin{lemma}[Definition of $b_{i,t}, j_{i,t}$]\label{lem: def bit, jit}
	For any $\gamma_i\in\M$, any $0\leq t \leq T$, there exist $b_{i,t} \in\{-1, 1\}$ and $1\leq  j_{i,t} \leq n_x$ such that $b_{i,t}, j_{i,t} \in \F(z_t)\subseteq \F_t$ and
	$$\|\gamma_i z_t\|_\infty = b_{i,t} (\gamma_i z_t)^{j_{i,t}}.$$
\end{lemma}
	Note that one way to determine $b_{i,t}, j_{i,t} $ from $z_t$ is by the following: first pick the smallest $j$ such that $|(\gamma_i z_t)^j|=\|\gamma_i z_t\|_\infty$, then let $b_{i,t}=\textup{sgn}((\gamma_i z_t)^j)$, where $\textup{sgn}(\cdot)$ denotes the sign of a scalar argument.
\begin{proof}
	For any $\gamma_i\in \M$, any $0\leq t \leq T$, we have
	$$\|\gamma_i z_t\|_\infty =\max_{1\leq j \leq n_x}\max_{b\in\{-1,1\}} b(\gamma_i z_t)^j$$
	Hence, there exist $b_{i,t}, j_{i,t}$ such that $\|\gamma_i z_t\|_\infty=b_{i,t} (\gamma_i z_t)^{j_{i,t}}$. Further, $b_{i,t}, j_{i,t}$ only depend on $\gamma_i$ and $z_t$, so they are $\F(z_t)$-measurable, and $\F(z_t)\subseteq \F_t$.
\end{proof}

\begin{lemma}[Definition of stopping times $L_{i,k}$]\label{lem: define stopping times Lik}
	Let $\eta=\frac{a_1}{\sqrt{n_x}}$. For any $\gamma_i\in\M$, any $0\leq k \leq T/m-1$, we can define a random time index $1\leq L_{i,k}\leq m+1$ by
	$$L_{i,k}=\min(m+1,\min\{ l\geq 1: \|\gamma_i z_{km+l}\|_\infty \geq \eta \}).$$
	Then, we have $1\leq L_{i,k}\leq m+1$. Further, for any $1\leq l \leq m$, $\{L_{i,k}=l\} \in \F_{km+l}$, and $\{L_{i,k}=m+1\}\in \F_{km+m}\subseteq \F_{km+m+1}$.
	In other words, $L_{i,k}$ is a stopping time with respect to filtration $\{F_{km+l}\}_{l\geq 1}$.
\end{lemma}
\begin{proof}
	For any $i$ and any $k$, it is straightforward to see that $L_{i,k}$ is well-defined and $1\leq L_{i,k}\leq m+1$.
	
	When $L_{i,k}=l\leq m$, this is equivalent with $\|\gamma_i z_{km+l}\|_\infty \geq \eta$ but $\|\gamma_i z_{km+s}\|<\eta$ for $1\leq s<l$. Notice that this event is only determined by $z_{km+l}, \dots, z_{km+1}$, so $\{L_{i,k}=l\} \in \F_{km+l}$.
	
	When $L_{i,k}=m+1$, this is equivalent with $\|\gamma_i z_{km+s}\|<\eta$ for $1\leq s\leq m$. Notice that this event is only determined by $z_{km+m}, \dots, z_{km+1}$, so $\{L_{i,k}=m+1\} \in \F_{km+m}$.
	
	Therefore, by definition, $L_{i,k}$ is a stopping time with respect to filtration $\{\F_{km+l}\}_{l\geq 1}$.
\end{proof}

\subsubsection{Step (ii): a formal version of Lemma \ref{lem: discretize E1 E2} and its proof}
\label{sec:D32}

\begin{lemma}[Discretization of $\Ec_1\cap\Ec_2$ (Formal version of Lemma \ref{lem: discretize E1 E2})]\label{lem: discretize E1 E2 formal}
	Let $\M=\{\gamma_1, \dots, \gamma_{v_\gamma}\}$ be an $\epsilon_{\gamma}$-net of $\{\gamma: \|\gamma\|_F=1\}$ as defined in Lemma \ref{lem: mesh on unit sphere of matrices}, where $\epsilon_{\gamma}=\min(\frac{a_1}{4b_z \sqrt n_x}, 1)$, $v_\gamma = \tilde O(n_x^{2.5}n_z^{2.5})a_4^{n_x n_z}$, and $a_4=\frac{4b_z \sqrt{n_x}}{a_1}$. Define
	\begin{align*}
		\Ec_{1,i}=\{ \exists \, \gamma\in \Gamma_T, \text{ s.t. } b_{i, km+L_{i,k}}(\gamma z_{km+L_{i,k}})^{j_{i, km+L_{i,k}}} \geq \frac{a_1\delta}{4\sqrt n_x}, \ \forall\, k\geq 0\}.
	\end{align*}
	Then, we have
	$$\Pb(\Ec_1\cap\Ec_2)\leq \sum_{i=1}^{v_\gamma} \Pb(\Ec_{1,i}\cap \Ec_2).$$

\end{lemma}

The rest of this subsubsection is dedicated to the proof of Lemma \ref{lem: discretize E1 E2 formal}. As an overview: firstly, we will discuss the implications of $\Ec_2$ on $\gamma_i\in \M$. Then, we discuss the implications of $\Ec_2$ on any $\gamma$. Lastly, we prove Lemma \ref{lem: discretize E1 E2 formal} by combining the implications of $\Ec_2$ on any $\gamma$ and $\|\gamma\|_F\geq \delta/2$.

\begin{lemma}[The implication of $\Ec_2$ on $\gamma_i$]\label{lem: implication of E2 on gammai}
	If $\Ec_2$ happens, then for any $\gamma_i \in \M$, any $0\leq k \leq T/m-1$, we have
	$$\max_{1\leq s\leq m}\|\gamma_iz_{km+s}\|_\infty \geq  \frac{a_1}{\sqrt n_x}.$$
	Therefore, almost surely, we have $1\leq L_{i,k}\leq m$ and
		$$b_{i, km+L_{i,k}}(\gamma_iz_{km+L_{i,k}})^{j_{i, km+L_{i,k}}} \geq \frac{a_1}{\sqrt n_x}.$$
\end{lemma}
\begin{proof}
	If $\Ec_2$ happens, then by definition, we have
		$$\frac{1}{m}\sum_{{s=1}}^{{m}} {z_{km+s} z_{km+s}^\top} \succeq  a_1^2 I_{{n_z}}, $$
	for all $0\leq k \leq T/m-1$.

	Now, for any $\gamma_i\in\M$, we have that
	\begin{align}\label{equ: Di xt psd}
		\frac{1}{m}\sum_{s=1}^{m} \gamma_iz_{km+s}z_{km+s}^\top \gamma_i^\top \succeq  a_1^2 \gamma_i \gamma_i^\top.
	\end{align}
	
	Therefore, by taking trace at each side of \eqref{equ: Di xt psd}, we obtain
	\begin{align}
		\frac{1}{m}\sum_{s=1}^{m} \tr(\gamma_iz_{km+s}z_{km+s}^\top \gamma_i^\top)\geq a_1^2 \tr(\gamma_i \gamma_i^\top)
	\end{align}
	Since $\gamma_i \in \Sb_{n_x\times n_z}(0,1)$, we have $\|\gamma_i\|_F=1$, so $\tr(\gamma_i \gamma_i^\top)=\tr( \gamma_i^\top \gamma_i)=\|\gamma_i\|_F^2=1$. Further, we have 
	$$\tr(\gamma_iz_{km+s}z_{km+s}^\top \gamma_i^\top)=\tr(z_{km+s}^\top \gamma_i^\top \gamma_iz_{km+s} )=z_{km+s}^\top \gamma_i^\top \gamma_iz_{km+s}=\|\gamma_iz_{km+s}\|_2^2.$$
	Consequently, we have
	$$	\frac{1}{m}\sum_{s=1}^{m} \|\gamma_iz_{km+s}\|_2^2 \geq a_1^2$$
	for all $k$.
	
	By the   pigeonhole principle, we have that 
	$$\max_{1\leq s \leq m} \|\gamma_iz_{km+s}\|_2^2 \geq a_1^2.$$
	This is equivalent with $\max_{1\leq s \leq m} \|\gamma_iz_{km+s}\|_2 \geq a_1.$
	
	Notice that $\|\gamma_iz_{km+s}\|_2 \leq \sqrt{n_x} \|\gamma_iz_{km+s}\|_\infty$, 
	so $ \max_{1\leq s \leq m} \sqrt{n_x}\|\gamma_iz_{km+s}\|_\infty \geq a_1$, which completes the proof of the first inequality in the lemma statement.

	Next, we prove the second inequality in the lemma statement. Notice that by the definition of $L_{i,k}$ in Lemma \ref{lem: define stopping times Lik} and by $\eta=\frac{a_1}{\sqrt{n_x}}$, we have $1\leq L_{i,k}\leq m$ and  $\|\gamma_iz_{km+L_{i,k}}\|_\infty \geq  \frac{a_1}{\sqrt n_x}$ for all $k$. Further, 
	by Lemma \ref{lem: def bit, jit}, we have $\|\gamma_iz_{km+L_{i,k}}\|_\infty= b_{i, km+L_{i,k}}(\gamma_iz_{km+L_{i,k}})^{j_{i, km+L_{i,k}}}$ almost surely. Hence, we have $b_{i, km+L_{i,k}}(\gamma_iz_{km+L_{i,k}})^{j_{i, km+L_{i,k}}}\geq  \frac{a_1}{\sqrt n_x}$, which completes the proof.

%
%
\end{proof}

\begin{lemma}[The implication of $\Ec_2$ on $\gamma z_t$]\label{lem: implication of E2 on gamma z}
	If $\Ec_2$ happens, then for any $\gamma\in \R^{n_x\times n_z}$, there exists $1\leq i \leq v_\gamma$, such that 
	$$b_{i, km+L_{i,k}}(\gamma z_{km+L_{i,k}})^{j_{i, km+L_{i,k}}} \geq \frac{a_1}{2\sqrt n_x}\|\gamma\|_F,$$
	for all $0\leq k \leq T/m-1$.
\end{lemma}
\begin{proof}
	Firstly, when $\gamma=0$, the inequality holds because both sides are 0.
	
	Next, when $\gamma\not =0$,  it suffices to prove $b_{i, km+L_{i,k}}(\frac{\gamma}{\|\gamma\|_F} z_{km+L_{i,k}})^{j_{i, km+L_{i,k}}} \geq \frac{a_1}{2\sqrt n_x}$. Therefore, we will only consider $\gamma\in \Sb_{n_x\times n_z}(0,1)$. By Lemma \ref{lem: mesh on unit sphere of matrices}, there exists $\gamma_i \in \M$ such that $\|\gamma-\gamma_i \|_F \leq 2\epsilon_\gamma =\min(\frac{a_1}{2b_z \sqrt n_x}, 2)$. Notice that by Lemma \ref{lem: implication of E2 on gammai}, if $\Ec_2$ happens, for all $k$, we have
	$$b_{i, km+L_{i,k}}(\gamma_i z_{km+L_{i,k}})^{j_{i, km+L_{i,k}}} \geq \frac{a_1}{\sqrt n_x}.$$
	Therefore, 
	\begin{align*}
		b_{i, km+L_{i,k}}(\gamma z_{km+L_{i,k}})^{j_{i, km+L_{i,k}}}& = b_{i, km+L_{i,k}}(\gamma_i z_{km+L_{i,k}})^{j_{i, km+L_{i,k}}}\\
  &\quad - b_{i, km+L_{i,k}}((\gamma_i-\gamma) z_{km+L_{i,k}})^{j_{i, km+L_{i,k}}}\\
		& \geq \frac{a_1}{\sqrt n_x}- |b_{i, km+L_{i,k}}((\gamma_i-\gamma) z_{km+L_{i,k}})^{j_{i, km+L_{i,k}}}|\\
		& \geq \frac{a_1}{\sqrt n_x}- \|(\gamma_i-\gamma) z_{km+L_{i,k}}\|_2\\
		& \geq  \frac{a_1}{\sqrt n_x}-\|\gamma_i -\gamma\|_2 \|z_{km+L_{i,k}}\|_2\\
		& \geq \frac{a_1}{\sqrt n_x}-  2\epsilon_\gamma b_z\geq \frac{a_1}{2\sqrt n_x}
	\end{align*}
	
\end{proof}

\begin{proof}[\textbf{Proof of Lemma \ref{lem: discretize E1 E2 formal}}]
	By Lemma \ref{lem: implication of E2 on gamma z}, 	under $\Ec_2$,  for any $\gamma\in \R^{n_x\times n_z}$, there exists $1\leq i \leq v_\gamma$, such that 
	$$b_{i, km+L_{i,k}}(\gamma z_{km+L_{i,k}})^{j_{i, km+L_{i,k}}} \geq \frac{a_1}{2\sqrt n_x}\|\gamma\|_F,$$
	for all $0\leq k \leq T/m-1$. Therefore, if $\Ec_1\cap \Ec_2$ happens, there exists $\gamma \in \Gamma_T$ and a corresponding $i$, such that 
	$$b_{i, km+L_{i,k}}(\gamma z_{km+L_{i,k}})^{j_{i, km+L_{i,k}}} \geq \frac{a_1}{2\sqrt n_x}\|\gamma\|_F\geq \frac{a_1\delta}{4\sqrt n_x}.$$
	Therefore,
	\begin{align*}
		\Pb(\Ec_1\cap \Ec_2)\leq \Pb(\bigcup_{i=1}^{v_\gamma} \Ec_{1,i} \cap \Ec_2)\leq \sum_{i=1}^{v_\gamma}\Pb(\Ec_{1,i} \cap \Ec_2),
	\end{align*}
which completes the proof.
\end{proof}

\subsubsection{Proof of Lemma \ref{lem: construct Gik and Lik}}
\label{sec:D33}
Notice that Lemma \ref{lem: construct Gik and Lik} states two inequalities: in the following, we will first prove the first inequality $\Pb(\Ec_{1,i}\cap \Ec_2)\leq \Pb(\cap_{k=0}^{T/m-1} G_{i,k})$, then prove the second inequality on $\Pb(G_{i,k}\mid \cap_{k'=0}^{k-1}G_{i,k'})$.

\begin{lemma}[Bound $\Ec_{1,i}\cap \Ec_2$ by $G_{i,k}$]\label{lem: bdd Ec1i , E2, by Gik}
	Under the conditions in Lemma \ref{lem: construct Gik and Lik}, for any $1\leq i \leq v_{\gamma}$, we have
	$$\Pb(\Ec_{1,i}\cap \Ec_2)\leq \Pb(\bigcap_{k=0}^{T/m-1} G_{i,k}).$$
\end{lemma}
\begin{proof}
	Firstly, for any $\gamma\in\Gamma_T$, we have $\|w_t -\gamma z_t\|_\infty \leq w_{\max}$ for all $t \geq 0$. This suggests that, for any $1\leq j\leq n_x$, we have
	$$-w_{\max}\leq  w_t^j - (\gamma z_t)^j\leq w_{\max}. $$
	Hence, we have $b(\gamma z_t)^j \leq bw_t^j+w_{\max}$ for any $b\in\{-1,1\}$, $1\leq j\leq n_x$, and $t\geq 0$.
	
	Next, by $\Ec_{1,i}$, there exists $\gamma\in \Gamma_T$ such that $b_{i, km+L_{i,k}}(\gamma z_{km+L_{i,k}})^{j_{i, km+L_{i,k}}} \geq \frac{a_1\delta}{4\sqrt n_x}$ for all $k\geq 0$. Therefore, $b_{i, km+L_{i,k}}  w_{km+L_{i,k}}^{j_{i,km+L_{i,k}}} +w_{\max}\geq  \frac{a_1\delta}{4\sqrt{n_x}}$ for all $k$.
	
	Finally, $\Ec_{1,i}\cap \Ec_2$ implies that $b_{i, km+L_{i,k}}  w_{km+L_{i,k}}^{j_{i,km+L_{i,k}}} +w_{\max}\geq  \frac{a_1\delta}{4\sqrt{n_x}}$ and $\frac{1}{m}\sum_{{s=1}}^{{m}} {z_{km+s} z_{km+s}^\top} \succeq  a_1^2 I_{{n_z}}$ for all $k$, which is $\bigcap_k G_{i,k}$ by the definition of $G_{i,k}$.

\end{proof}

\begin{lemma}[Bound on $\Pb(G_{i,k}\mid \cap_{k'=0}^{k-1}G_{i,k'})$]\label{lem: bound Pr(Gik|Gik' for k'<k)}
		Under the conditions in Lemma \ref{lem: construct Gik and Lik}, for any $1\leq i \leq v_{\gamma}$ and any $k\geq 0$, we have
			$$ \Pb(G_{i, k}\mid \bigcap_{k'=0}^{k-1}G_{i,k'})\leq  1-q_w(\frac{a_1 \delta}{4\sqrt n_x}).$$
\end{lemma}

\begin{proof}
	Firstly, notice that when $\frac{1}{m}\sum_{{s=1}}^{{m}} {z_{km+s} z_{km+s}^\top} \succeq  a_1^2 I_{{n_z}}$, we have $1\leq L_{i,k}\leq m$ by the proof of Lemma \ref{lem: implication of E2 on gammai}. Therefore, we have
	\begin{align*}
		& \Pb(G_{i, k}\mid \bigcap_{k'=0}^{k-1}G_{i,k'}) \leq \Pb(b_{i, km+L_{i,k}}  w_{km+L_{i,k}}^{j_{i,km+L_{i,k}}} +w_{\max}\geq  \frac{a_1\delta}{4\sqrt{n_x}}, 1\leq L_{i,k}\leq m \mid  \bigcap_{k'=0}^{k-1}G_{i,k'})\\
		  \leq\ & \sum_{l=1}^m \Pb(b_{i, km+l}  w_{km+l}^{j_{i,km+l}} +w_{\max}\geq  \frac{a_1\delta}{4\sqrt{n_x}},  L_{i,k}=l \mid  \bigcap_{k'=0}^{k-1}G_{i,k'})\\
		  \leq  \ & \sum_{l=1}^m \Pb(b_{i, km+l}  w_{km+l}^{j_{i,km+l}} +w_{\max}\geq  \frac{a_1\delta}{4\sqrt{n_x}}\mid L_{i,k}=l,   \bigcap_{k'=0}^{k-1}G_{i,k'}) \Pb(L_{i,k}=l \mid  \bigcap_{k'=0}^{k-1}G_{i,k'})\\
		   \overset{(a)}{\leq} \ & (1-q_w(\frac{a_1 \delta}{4\sqrt n_x}))\sum_{l=1}^m\Pb(L_{i,k}=l \mid  \bigcap_{k'=0}^{k-1}G_{i,k'})\\
		    \leq \ & 1-q_w(\frac{a_1 \delta}{4\sqrt n_x})
	\end{align*}
The inequality $(c)$ is proved in the following:
\begin{align*}
& 	\Pb(b_{i, km+l}  w_{km+l}^{j_{i,km+l}} +w_{\max}\geq  \frac{a_1\delta}{4\sqrt{n_x}}\mid L_{i,k}=l,   \bigcap_{k'=0}^{k-1}G_{i,k'}) \\
= \ & \int_{v_{0:km+l}}\Pb(b_{i, km+l}  w_{km+l}^{j_{i,km+l}} +w_{\max}\geq  \frac{a_1\delta}{4\sqrt{n_x}},  w_{0:km+l}=v_{0:km+l}\mid L_{i,k}=l,   \bigcap_{k'=0}^{k-1}G_{i,k'})  \mathrm d v_{0:km+l}\\
=\ & \int_{v_{0:km+l}\in S_{km+l}} \ \Pb(b_{i, km+l}  w_{km+l}^{j_{i,km+l}} +w_{\max}\geq  \frac{a_1\delta}{4\sqrt{n_x}}\mid w_{0:km+l}=v_{0:km+l}) \\
& \qquad\qquad\qquad \ \  \times \Pb( w_{0:km+l}=v_{0:km+l}\mid L_{i,k}=l,   \bigcap_{k'=0}^{k-1}G_{i,k'})\mathrm d v_{0:km+l}\\
 \overset{(b)}{\leq}  \  & (1-q_w(\frac{a_1 \delta}{4\sqrt n_x})) \int_{v_{0:km+l}\in S_{km+l}}\Pb( w_{0:km+l}=v_{0:km+l}\mid L_{i,k}=l,   \bigcap_{k'=0}^{k-1}G_{i,k'})\mathrm d v_{0:km+l}\\\\
= \ & 1-q_w(\frac{a_1 \delta}{4\sqrt n_x}),
\end{align*}
 where we define a shorthand notation $w_{0:km+l} =(w_0, \dots, w_{km+l-1})$, and we use $v_{0:km+l}$ to denote a realization of $w_{0:km+l} $, then we define the set of values of $w_{0:km+l}$ as $S_{km+l}$ such that $L_{i,k}=l,   \bigcap_{k'=0}^{k-1}G_{i,k'}$ holds. Notice that $L_{i,k}=l$ can be determined by a set of values of $w_{0:km+l}$ because $L_{i,k}$ is a stopping time of $\{F_{km+l}\}_{l\geq 1}$ and thus $\{L_{i,k}=l\}\in \F_{km+l}$. The inequality $(b)$ above is because of the following: firstly, notice that $b_{i,km+l}, j_{i,km+l}\in \F_{km+l}$, so $b_{i,km+l}, j_{i,km+l}$ are deterministic values when $w_{0:km+l}=v_{0:km+l}$. Further, since $w_{km+l}$ is independent of $w_{0:km+l}$, we have $\Pb(w_{max}+b w_{km+l}^j \geq \epsilon\mid w_{0:km+l}=v_{0:km+l}) \leq 1-q_w(\epsilon)$ for any deterministic $b, j$ and any $\epsilon>0$ by  Assumption \ref{ass: tight bound on wt}. Hence, we have $\Pb(b_{i, km+l}  w_{km+l}^{j_{i,km+l}} +w_{\max}\geq  \frac{a_1\delta}{4\sqrt{n_x}}\mid w_{0:km+l}=v_{0:km+l})\leq 1-q_w(\frac{a_1 \delta}{4\sqrt n_x})$.
\end{proof}

\subsubsection{Proof of Lemma \ref{lem: bound PE and gamma>delta/2}}
The proof is by leveraging Lemma \ref{lem: discretize E1 E2 formal} and Lemma \ref{lem: construct Gik and Lik}.
\begin{align*}
	\Pb(\Ec_1\cap \Ec_2)& \leq \sum_{i=1}^{v_\gamma} \Pb(\Ec_{1,i}\cap \Ec_2)\\
	& \leq \sum_{i=1}^{v_\gamma} \Pb(\bigcap_{k=0}^{T/m-1} G_{i,k})\\
	& = \sum_{i=1}^{v_\gamma} \Pb(G_{i,0}) \Pb(G_{i,1}\mid G_{i,0})\cdots \Pb(G_{i, T/m-1}\mid \bigcap_{k=0}^{T/m-2}G_{i,k})\\
	& \leq \sum_{i=1}^{v_\gamma}  (1-q_w(\frac{a_1 \delta}{4\sqrt n_x}))^{T/m}\\
	& \leq  544 n_x^{2.5}n_z^{2.5} \log(a_4 n_x n_z)a_4^{n_zn_x}(1-q_w(\frac{a_1 \delta}{4\sqrt n_x}))^{T/m}.
\end{align*}

\subsection{A more precise upper bound for Theorem \ref{thm: estimation err bdd}}\label{append: precise upper bdd}
By the proof of Lemma \ref{lem: bound PE c} and Lemma \ref{lem: bound PE and gamma>delta/2} above, we have
\begin{equation}
    \Pb(\diam(\Theta_T)>\delta)\leq 544 \frac{T}{m}  n_z^{2.5} \log(a_2 n_z)a_2^{n_z} \exp(-a_3 m )+544 n_x^{2.5}n_z^{2.5} \log(a_4 n_x n_z)a_4^{n_zn_x}(1-q_w(\frac{a_1 \delta}{4\sqrt n_x}))^{T/m}\label{equ: diam bdd explicit}
\end{equation}

\section{Proof of Corollary \ref{cor: uniform dist est. bdd}}
\label{sec:cor1}
The proof involves two parts. Firstly, we will show that $\text{Term 1}\leq \epsilon$ under our choice of $m$. Secondly, we will let $\text{Term 2}=\epsilon$, then we will show $\delta \leq \tilde O(n_x^{1.5}n_z^2/T)$, which completes the proof.

\nbf{Step 1: show $\text{Term 1}\leq \epsilon$.} Notice that when $m\geq \frac{1}{a_3}(\log(\frac{T}{\epsilon})+n_z\log(a_2) + 2.5 \log(n_z) + \log\log(a_2n_z)+7)= O(n_z+ \log T+\log(1/\epsilon))$, we have $ T \tilde O(n_z^{2.5}) a_2^{n_z} \exp(-a_3m)\leq \epsilon$. Since $m\geq 1$, we obtain $\text{Term 1}\leq \epsilon$.


\nbf{Step 2: let $\text{Term 2}=\epsilon$ and show $\delta \leq \tilde O(n_x^{1.5}n_z^2/T)$.} Let
$\text{Term 2}=\epsilon$, then we have
$ (1-q_w(\frac{a_1 \delta}{4\sqrt n_x}))^{T/m} = \frac{\epsilon}{\tilde O(n_x^{2.5}n_z^{2.5})a_4^{n_x n_z}}$. Then, we obtain $(1-q_w(\frac{a_1 \delta}{4\sqrt n_x})) =\left( \frac{\epsilon}{\tilde O(n_x^{2.5}n_z^{2.5})a_4^{n_x n_z}}\right)^{m/T}$, which is equivalent with 
$$q_w(\frac{a_1 \delta}{4\sqrt n_x})=1-\left( \frac{\epsilon}{\tilde O(n_x^{2.5}n_z^{2.5})a_4^{n_x n_z}}\right)^{m/T}.$$
When $q_w(\frac{a_1 \delta}{4\sqrt n_x})= O(\frac{a_1 \delta}{4\sqrt n_x})$, we obtain
\begin{align*}
    \delta & = O(\frac{4\sqrt n_x}{a_1})\left(1-\left( \frac{\epsilon}{\tilde O(n_x^{2.5}n_z^{2.5})a_4^{n_x n_z}}\right)^{m/T}\right)\\
    & \leq O(\frac{-4\sqrt n_x}{a_1})\log\left(\left(\frac{\epsilon}{\tilde O(n_x^{2.5}n_z^{2.5})a_4^{n_x n_z}}\right)^{m/T}\right)\\
    & = O(\frac{4\sqrt n_x m}{a_1T})(\log(1/\epsilon) + n_xn_z+ \log(n_xn_z))\\
    & = \tilde O\left( \frac{n_x^{1.5}n_z^2}{T}\right).
\end{align*}

\nbf{Step 3: prove Corollary \ref{cor: uniform dist est. bdd}.}
By leveraging the bounds above and Theorem \ref{thm: estimation err bdd}, we have $\Pb(\diam(\Theta_T)\leq \tilde O\left( \frac{n_x^{1.5}n_z^2}{T}\right)) \geq \Pb(\diam(\Theta_T)\leq \delta)\geq 1-2\epsilon$. 

Since $\theta^*\in \Theta_T$ by definition, for any $\hat \theta_T\in \Theta_T$, we have $\|\hat \theta_T-\theta^*\|_F \leq \diam(\Theta_T)\leq \tilde O\left( \frac{n_x^{1.5}n_z^2}{T}\right)$ with probability at least $1-2\epsilon$.

\section{Proof of Corollary \ref{cor: estimation err bdd for B=0}}
\label{sec:cor2}
We provide a formal version of Corollary \ref{cor: estimation err bdd for B=0} and its proof below.

\begin{corollary}[Convergence rate when $B^*=0$ (formal version)]
	When $A^*$ is $(\kappa, \rho)$-stable, i.e.,   $\|(A^*)^t\|_2\leq \kappa (1-\rho)^t$ for all $t$ with $\rho <1$, for any $m>0$ and any $\delta>0$, when $T>m$, we have
	\begin{align*}
		\Pb(\diam(\mathbb A_{ {T}})>\delta)\leq \frac{T}{m} \tilde O(n_x^{2.5}) a_2^{n_x}\exp(-a_3m)+\tilde O(n_x^5)a_4^{n^2}(1-q_w(\frac{a_1\delta}{4\sqrt n_x}))^{\ceil{T/m}}  
	\end{align*}
	where $b_x=\kappa\|x_0\|_2+\kappa\sqrt{n_x}/\rho$, $p_x=1/192$, $\sigma_x=\sqrt{\lambda_{\min}(\Sigma_w)/2}$, $a_1 = \frac{\sigma_{x} p_{x}}{4}$, $a_2=\frac{64 w_{\max}}{\sigma_{x}^2 p_{x}^2}$, $a_3= \frac{p_{x}^2}{8}$, $a_4=\frac{4b_x \sqrt{n_x}}{a_1}$.

 Consequently, when the distribution of $w_t$ satisfies $q_w(\epsilon)=O(\epsilon)$, e.g. uniform or truncated Gaussian, we have
 $\|\hat \theta -\theta_*\|\leq \tilde O(n_x^{3.5}/T)$. 
	
\end{corollary}

The proof of Corollary \ref{cor: estimation err bdd for B=0} is exactly the same as the proofs of Theorem \ref{thm: estimation err bdd} and Corollary \ref{cor: uniform dist est. bdd}. When $A^*$ is stable, we can show that $\|x_t\|_2 \leq b_x$ for all $t$. Further, by \cite{dean2019safely}, the sequence $\{x_t\}_{t\geq 0}$ satisfies the $(1, \sigma_x, p_x)$-BMSB condition. Therefore, we complete the proof.

\section{Proof of Theorem \ref{thm: loose bound}}
\label{sec:thrm2}
Specifically, we define $\epsilon_0= \frac{4\sqrt n_x}{a_1}(\hat w_{\max}-w_{\max})$.

The proof is similar to the proof of Theorem \ref{thm: estimation err bdd}.  Firstly, we define $\hat \Gamma_T$ as a translation of the set $\hat \Theta_T$:
\begin{align}\label{equ: define hat Gamma t}
\hat	\Gamma_t=\bigcap_{s=0}^{t-1} \{\gamma : \|w_s -\gamma  z_s \|_\infty \leq \hat w_{\max}\}, \quad \forall\, t\geq 0.
\end{align}
Notice that 
$$\hat \Theta_T=\theta^*+\hat	\Gamma_T$$
by considering $\gamma=\hat\theta -\theta^*$. Therefore, we can upper bound our goal event $\{\diam(\hat \Theta_T)>\delta+\epsilon_0\}$ by the event $\Ec_3$ defined below.
\begin{align}
	\label{eq:middle hat}
	\Pb(\diam(\hat \Theta_T)>\delta +\epsilon_0) \leq \Pb(\Ec_3), \text{ where } \Ec_3:=\{\exists\, \gamma \in \hat \Gamma_T, \text{ s.t. } \|\gamma\|_F\geq \frac{\delta+\epsilon_0}{2} \}.
\end{align}

Next, notice that 
$$ 	\Pb(\diam(\hat \Theta_T)>\delta +\epsilon_0) \leq \Pb(\Ec_3)\leq \Pb(\Ec_3\cap \Ec_2)+\Pb(\Ec_2^c)$$
By Lemma \ref{lem: bound PE c}, we have already shown $\Pb(\Ec_2^c)\leq \text{Term 1} $. So we only need to discuss $\Pb(\Ec_3\cap \Ec_2)$.

\begin{lemma}\label{lem: bdd Ec2 Ec3}
	$$\Pb(\Ec_3\cap \Ec_2)\leq \textup{Term 2}$$
\end{lemma}
\begin{proof}
Firstly, 	define 	
	\begin{align*}
		\Ec_{3,i}=\{ \exists \, \gamma\in \hat \Gamma_T, \text{ s.t. } b_{i, km+L_{i,k}}(\gamma z_{km+L_{i,k}})^{j_{i, km+L_{i,k}}} \geq \frac{a_1(\delta+\epsilon_0)}{4\sqrt n_x}, \ \forall\, k\geq 0\}.
	\end{align*}
	We have
	$\Pb(\Ec_3\cap\Ec_2)\leq \sum_{i=1}^{v_\gamma} \Pb(\Ec_{3,i}\cap \Ec_2)$
 based on the same proof ideas of Lemma \ref{lem: discretize E1 E2 formal}.

 Next, we will show that \begin{align}
     \Pr(\Ec_{3,k}\cap \Ec_2)\leq \Pb(\bigcap_{k=0}^{T/m-1}G_{i,k}) \label{equ: E3k by Gik}
 \end{align}
 This is because for any $\gamma\in\hat \Gamma_T$, we have  $b(\gamma z_t)^j \leq bw_t^j+\hat w_{\max}$ for any $b\in\{-1,1\}$, $1\leq j\leq n_x$, and $t\geq 0$. By $\Ec_{3,i}$, there exists $\gamma\in \hat \Gamma_T$ such that $b_{i, km+L_{i,k}}(\gamma z_{km+L_{i,k}})^{j_{i, km+L_{i,k}}} \geq \frac{a_1(\delta+\epsilon_0)}{4\sqrt n_x}$ for all $k\geq 0$. Thus, $b_{i, km+L_{i,k}}  w_{km+L_{i,k}}^{j_{i,km+L_{i,k}}} +\hat w_{\max}\geq  \frac{a_1(\delta+\epsilon_0)}{4\sqrt{n_x}}$ for all $k$.  Notice that this is equivalent with $b_{i, km+L_{i,k}}  w_{km+L_{i,k}}^{j_{i,km+L_{i,k}}} + w_{\max}\geq  \frac{a_1 \delta}{4\sqrt{n_x}}$ for all $k$ because
	 $\epsilon_0= \frac{4\sqrt n_x}{a_1}(\hat w_{\max}-w_{\max})$. In this way, we can prove \eqref{equ: E3k by Gik}.
  

	Finally, we can complete the proof by the following.
	\begin{align*}
		\Pb(\Ec_3\cap \Ec_2)& \leq \sum_{i=1}^{v_\gamma} \Pb(\Ec_{3,i}\cap \Ec_2) \leq \sum_{i=1}^{v_\gamma} \Pb(\bigcap_{k=0}^{T/m-1} G_{i,k})\\
		& = \sum_{i=1}^{v_\gamma} \Pb(G_{i,0}) \Pb(G_{i,1}\mid G_{i,0})\cdots \Pb(G_{i, T/m-1}\mid \bigcap_{k=0}^{T/m-2}G_{i,k})\\
		& \leq \sum_{i=1}^{v_\gamma}  (1-q_w(\frac{a_1 \delta}{4\sqrt n_x}))^{T/m} \leq \text{Term 1}
	\end{align*}
 where the second last inequality is by Lemma \ref{lem: bound Pr(Gik|Gik' for k'<k)} and the last inequality uses the definition of $v_\gamma$ in Lemma \ref{lem: mesh on unit sphere of matrices}.
\end{proof}


\section{Proofs of Theorem \ref{thm: convergence rate of wbarmax}, \Cref{cor: wbarmax bdd when qw=O(epsilon)}, and \Cref{thm: convergence of ucb-sme}}\label{append: unknown wmax}

This section provides proofs of the main results related to the SME with unknown $w_{\max}$ as discussed in \cref{sec: unknow wmax}. Namely, Theorem \ref{thm: convergence rate of wbarmax} and \Cref{cor: wbarmax bdd when qw=O(epsilon)} provide the rate of convergence of the estimator $\wbarmax$ defined in \eqref{equ: define wbarmax} to $w_{\max}$, and \Cref{thm: convergence of ucb-sme} states the rate of convergence of UCB-SME algorithm introduced in \eqref{equ: ucb-sme}.

For ease of notation, we introduce the following function indexed by the time horizon $T>0$,
\begin{equation}\label{eq: W_T function for wbarmax}
	W_T : \theta \mapsto \max_{0\leq t \leq T-1} \| x_{t+1} - \theta z_t \|_\infty.
\end{equation}
The estimator $\wbarmax$ is simply the infimum of this function, i.e., $\wbarmax = \inf_{\theta} W_T(\theta)$.

\subsection{Proof of Theorem \ref{thm: convergence rate of wbarmax}}\label{append: conv rate wbarmax}

The proof of Theorem \ref{thm: convergence rate of wbarmax} involves two steps:
\begin{itemize}
    \item \textit{Step 1:} We demonstrate that the learning error of $w_{\max}$ incurred by the estimator $\wbarmax$ is governed by the diameter of the uncertainty set $\Theta_T$ and the minimum learning error achievable if $\theta^\ast$ were known.
    \item \textit{Step 2:} We then provide an upper bound the probability of learning error exceeding a fixed threshold.
\end{itemize}

Before we proceed with the proof of Theorem~\ref{thm: convergence rate of wbarmax}, we present the the following technical lemma.
\begin{lemma}\label{lem:properties of W_T}
    Consider the sequence of functions $ \{ W_T\}_{T>0} $ defined in \eqref{eq: W_T function for wbarmax}. The following holds:
	\begin{itemize}
	\item[i.] $W_T$ is convex in $\R^{n_x\times n_z}$,
        \item[ii.] The sequence $ \{\inf_{\theta } W_T(\theta)\}_{T>0} $ is bounded and monotonically non-decreasing, i.e., $$ 0\leq \inf_{\theta } W_T(\theta) \leq \inf_{\theta } W_{T+1}(\theta)  \leq w_{\max}, $$ for all $T>0$,
        \item[iii.] $W_T$ attains its minimum in $\Theta_T$, i.e., $ \argmin_{\theta} W_T(\theta) \subset \Theta_T$.
	\end{itemize}
\end{lemma}
\begin{proof}
$(i.)$ For $0\leq t\leq T-1$, the function $\theta \mapsto \| x_{t+1} - \theta z_t \|_\infty$ is convex due to convexity of norms. Since the maximum of convex functions is convex \cite{boyd2004convex}, convexity of $W_T$ follows.\\
$(ii.)$ Notice that $W_{T+1}$ can be defined in terms of $W_T$ recursively as $W_{T+1}(\theta) = \max\left(W_T(\theta), \,\| x_{T+1} - \theta z_T \|_\infty\right)$. Thus, $W_T(\theta) \leq W_{T+1}(\theta)$ for all $\theta\in \R^{n_x\times n_z}$, implying monotonicity of $ \{\inf_{\theta } W_T(\theta)\}_{T>0}$. To see boundedness, first notice that $$ W_T(\theta^*) = \max_{0\leq t \leq T-1} \|x_{t+1}-\theta^* z_t \|_\infty = \max_{0\leq t \leq T-1} \|w_t\|_\infty \leq w_{\max},$$ since $x_{t+1} = \theta^* z_t + w_t$. Therefore, for any $T>0$, we have that $$\inf_{\theta} W_T(\theta) = \inf_{\theta} \max_{0\leq t \leq T-1} \| x_{t+1} - \theta z_t \|_\infty \leq \max_{0\leq t \leq T-1} \| x_{t+1} - \theta^* z_t \|_\infty \leq w_{\max}$$.\\
$(iii.)$ First, we show that $W_T$ attains its minimum on $\R^{n_x\times n_z}$. If $z_t = 0$ for $t \in [T]$, then $W_T$ is a constant function and any $\theta \in \R^{n_x \times n_z}$ is a minimum of $W_T$. Now, suppose $z_t \neq 0$ for some $t \in [T]$. Then, $W_T$ diverges at the infinity, i.e.,  $\lim_{k\to \infty}W_T(\theta_k) = \infty$ for any sequence $\{\theta_k\}_{k\in \N}$ such that $\|\theta_k\|\to \infty$ as $k\to \infty$. Since $W_T$ is convex and bounded below with finite infimum, there exists a global minimizer $\bar{\theta}_T \in \R^{n_x \times n_z}$ such that $W_T(\bar{\theta}_T) = \inf_{\theta} W_T(\theta) = \wbarmax$. Furthermore, by $(ii)$,  we have that $\| x_{t+1} - \bar{\theta}_T z_t \|_\infty \leq w_{\max}$ for all $t\in[T]$ and any global minimizer $\bar{\theta}_T \in \argmin_{\theta} W_T(\theta)$, hence $\bar{\theta}_T \in \Theta_T$ by definition.
\end{proof}

\paragraph{\textit{Step 1 of the Proof of Theorem~\ref{thm: convergence rate of wbarmax}:}}
We first show that the error margin of the estimate $\wbarmax$ from $w_{\max}$ is governed by the sum of two factors: (i) the diameter of $\Theta_T$, which arises due to the lack of knowledge of $\theta^\ast$, and (ii) the minimum learning error achievable if $\theta^\ast$ were known, namely
\begin{equation}\label{eq: learning error bound}
0 \leq w_{\max}-\wbarmax \leq  {b_z \diam(\Theta_T)}+{w_{\max}-\max_{0\leq t \leq T-1} \|w_t \|_\infty}.
\end{equation}

First, $0 \leq w_{\max}-\wbarmax$ is simply due to Lemma~\ref{lem:properties of W_T}. Next, we prove the second inequality $w_{\max}-\wbarmax \leq b_z \diam(\Theta_T)+w_{\max}-\max_{0\leq t \leq T-1} \|w_t \|_\infty$. 
By Lemma~\ref{lem:properties of W_T}, there exists $\bar{\theta}_T \in \Theta_T$ such that $W_T(\bar{\theta}_T) = w_{\max}$ and 
	\begin{align*}
		w_{\max} & = \max_{0\leq t \leq T-1}\|x_{t+1}-\bar{\theta}_T z_t\|_\infty, \\
		& = \max_{0\leq t \leq T-1}\|x_{t+1}-\theta^* z_t+(\theta^*-\bar{\theta}_T)z_t\|_\infty,\\
		& \geq \max_{0\leq t \leq T-1}\left(\|x_{t+1}-\theta^* z_t\|_\infty-\|(\theta^*-\theta_k)z_t\|_\infty\right),
	\end{align*}
where the inequality is due to reverse triangle inequality. Furthermore, by using the equivalence of $\ell_2$ and $\ell_\infty$ norms, i.e., $\|x\|_2 \leq \|x\|_\infty$ for $x\in \R^{n_x}$, we bound $w_{\max}$ further below by
\begin{align*}
    w_{\max} & \geq  \max_{0\leq t \leq T-1}\left(\|x_{t+1}-\theta^* z_t\|_\infty-\|(\theta^*-\bar{\theta}_T)z_t\|_2\right),\\
		& \geq \max_{0\leq t \leq T-1}\left(\|x_{t+1}-\theta^* z_t\|_\infty-\|\theta^*-\bar{\theta}_T\|_2\|z_t\|_2\right),\\
		& \geq \max_{0\leq t \leq T-1}\|w_t\|_\infty -b_z \diam(\Theta_T),
\end{align*}
where the second inequality is due to $\|\theta^*-\bar{\theta}_T\|_2 \coloneqq \sup_{z\neq 0} \frac{\|(\theta^*-\bar{\theta}_T )z\|_2}{\|z\|_2} \leq \frac{\|(\theta^*-\bar{\theta}_T )z_t\|_2}{\|z_t\|_2} $ and the third inequality follows from the assumption $\|z_t\|_2\leq b_z$, the equivalence of Frobenius and spectral norms $\|\theta^* - \bar{\theta}_T\|_2\leq \|\theta^*-\bar{\theta}_T\|_F$, and $\theta^*, \bar{\theta}_T\in \Theta_T$. Consequently,
	$$w_{\max}-\wbarmax\leq w_{\max}-	\max_{0\leq t \leq T-1}\|w_t\|_\infty +b_z \diam(\Theta_T).$$
	This completes the proof of the first step. \qed

\paragraph{\textit{Step 2 of the Proof of Theorem~\ref{thm: convergence rate of wbarmax}:}} 
Using the learning error bound in \eqref{eq: learning error bound}, we obtain an upper bound on the probability of learning error exceeding a fixed $\delta>0$ as shown below
\begin{equation}
    \Pb(w_{\max}-\bar w_{\max}^{(T)}>\delta)\leq \T_1+\T_2\left(\frac{\delta}{2b_z}\right)+\T_5\left(\frac{\delta}{2}\right),
\end{equation}
where $\T_5(\delta)\coloneqq (1-q_w(\delta))^T$.

First, using the the fact that $\{w_t\}_{t=0}^{T-1}$ are iid, we show that 
\begin{align*}
	\Pb\left(w_{\max}-\max_{0\leq t \leq T-1} \|w_t\|_\infty>\delta\right)& = \Pb(w_{\max}-\delta> \|w_t\|_\infty, \ \forall\, 0\leq t \leq T-1),\\
	& = \prod_{t=0}^{T-1}\Pb(w_{\max}-\delta >\|w_t\|_\infty),\\
	& \leq  \prod_{t=0}^{T-1}\Pb(w_{\max}-\delta > w_t^1),\\
	& \leq (1-q_w(\delta))^T,
\end{align*}
where the first inequality is due to $w_t^1\leq \|w_t\|_\infty$ and the second inequality is from Assumption~\ref{ass: tight bound on wt}. Finally, we obtain the desired convergence rate using the error bound in \eqref{eq: learning error bound} as follows
\begin{align*}
	\Pb(w_{\max}-\wbarmax >\delta)& \leq \Pb\left(b_z\diam(\Theta_T)+w_{\max}-\max_{0\leq t \leq T-1}\|w_t\|_\infty>\delta\right)\\
	& \leq \Pb\left(b_z\diam(\Theta_T)>\delta/2\, \text{ or }\, w_{\max}-\max_{0\leq t \leq T-1}\|w_t\|_\infty>\delta/2\right)\\
	& \leq \Pb\left(\diam(\Theta_T)>\frac{\delta}{2b_z}\right)+ \Pb\left(w_{\max}-\max_{0\leq t \leq T-1}\|w_t\|_\infty>\delta/2\right)\\
	&\leq \T_1+\T_2\left(\frac{\delta}{2b_z}\right)+\T_5(\delta/2).
\end{align*}
where the last inequality is by Theorem~\ref{thm: estimation err bdd}.

This completes the second and the last step of the proof. \qed

\subsection{Proof of \Cref{cor: wbarmax bdd when qw=O(epsilon)}}
First, by the proof of Corollary \ref{cor: uniform dist est. bdd} in \cref{sec:cor1}, we have that $\T_1\!=\!\frac{T}{m} \tilde O(n_z^{2.5}) a_2^{n_z} \exp(-a_3m)\!\leq\! \epsilon$ whenever $m \!\geq\! O(n_z\!+\!\log T\!+\!\log\frac{1}{\epsilon})$.
	
Next, we show $\T_5(\delta_T/2)\leq \T_2(\frac{\delta_T}{2b_z})$. Since $b_z\geq \sigma_z$ by the definition of BMSB, we have $\frac{a_1\delta_T}{8\sqrt{n_x}b_z}\leq \frac{\delta_T}{2}$. Since $q_w(\cdot)$ is a non-decreasing function, we have $1-q_w(\frac{a_1\delta_T}{8\sqrt{n_x}b_z})\geq 1-q_w(\frac{\delta_T}{2})$. Notice that $m\geq 1$, and the constant factors in front of the $(1-q_w(\cdot))^{\ceil{(T/m)}}$ in $\T_2$ is also larger than 1.  Consequently, $\T_2(\frac{\delta_T}{2b_z})\geq \T_5(\delta/2)$. Therefore,  the  choice of $\delta_T$ for the second term $\T_2$ also guarantees $\T_5(\delta_T/2)\leq \epsilon$. 
 
  Therefore, it suffices to ensure $\T_2(\frac{\delta_T}{2b_z})\leq \epsilon$. 
	Notice that, when $\frac{\delta_T}{2b_z}=2w_{\max}$, then $\T_2(\frac{\delta_T}{2b_z})=0\leq \epsilon$, so there exists $\delta_T$ such that $\T_2(\frac{\delta_T}{2b_z})\leq \epsilon$. 
	
	Next, we will show that there exists such $\delta_T$ that diminishes to zero as $T$ goes to infinity. Notice that we need $$1-q_w\left(\frac{a_1\delta_T}{8b_z\sqrt{n_x}}\right)\leq \left(\frac{\epsilon}{\tilde O ((n_xn_z)^{2.5} a_4^{n_x n_z})}\right)^{1/\ceil{T/m}},$$ so that $$q_w\left(\frac{a_1\delta_T}{8b_z\sqrt{n_x}}\right)\geq 1- \left(\frac{\epsilon}{\tilde O ((n_xn_z)^{2.5} a_4^{n_x n_z})}\right)^{1/\ceil{T/m}},$$ where the right hand side converges to zero as $T\to \infty$.
 
 Now, consider $\delta(k)=1/k$. Since $q_w\left(\frac{a_1\delta(k)}{8b_z\sqrt{n_x}}\right)> 0$, there exists a large enough $T_k$ for any $k>0$ such that for any $T\geq T_k$, we have that 
 $$q_w\left(\frac{a_1\delta(k)}{8b_z\sqrt{n_x}}\right)\geq 1- \left(\frac{\epsilon}{\tilde O ((n_xn_z)^{2.5} a_4^{n_x n_z})}\right)^{1/\ceil{T_k/m}}.$$ 
 Furthermore, for any $T>0$, we can define 
 $$ \delta_T = \begin{cases}
     \delta(k), &\textrm{ if } T_k\leq T <T_{k+1}, \textrm{ for } k>0,\\
     2 w_{\max}, &\textrm{ if } T < T_1.
 \end{cases}$$
In this way, $\delta_T$ satisfies $\T_2(\frac{\delta_T}{2b_z})\leq \epsilon$ and $\delta_T\to 0$ as $T\to +\infty$.

Finally, using the proof of Corollary \ref{cor: uniform dist est. bdd}, we can show that there exists $\frac{\delta_T}{2b_z}=\tilde O(  n_x^{1.5}n_z^2/T)$ such that $\T_2(\frac{\delta_T}{2b_z})\leq \epsilon$ whenever $q_w(\delta)=O(\delta)$. This implies $\delta_T=\tilde O(  n_x^{1.5}n_z^2/T)$ and completes the proof. \qed

\subsection{Proof of \Cref{thm: convergence of ucb-sme}}
We first show that the unknown $\theta^*$ is a member of USC-SME uncertainty set $\hat \Theta^{\textup{ucb}}_T$ with high probability. By Theorem~\ref{thm: convergence rate of wbarmax}, Corollary \ref{cor: wbarmax bdd when qw=O(epsilon)}, and the definition in \eqref{equ: ucb-sme},  we have
	$$\Pb(w_{\max}>\whatmax)= \Pb(w_{\max}-\wbarmax >\delta_T)\leq 3\epsilon,$$
	which implies 
	$1-3\epsilon \leq \Pb(w_{\max}\leq \whatmax)\leq \Pb(\theta^* \in \hat \Theta^{\textup{ucb}}_T).$

Next, we show that the diameter of the UCB-SME uncertainty set is controlled by $\delta_T$ with high probability. Notice that $\hat \Theta^{\textup{ucb}}_T  \subseteq   \Theta_T(w_{\max}+\delta_T)$ because $\wbarmax \leq w_{\max}$. Therefore, by Theorem \ref{thm: loose bound},  the following holds for any constant $r>0$:
	\begin{align*}
		\Pb(	\diam(\hat \Theta^{\textup{ucb}}_T)> r+ a_5 \sqrt{n_x}\delta_T)& \leq \Pb(\diam(\Theta_T(w_{\max}+\delta_T))> r+ a_5 \sqrt{n_x}\delta_T),\\
		& \leq \T_1+\T_2(r).
	\end{align*}
Let $r=\delta_T$, then, using the inequality $\T_2(\delta_T)\leq \T_2(\delta_T/2b_z)$, we have that
\begin{align*}
	\Pb(	\diam(\hat \Theta^{\textup{ucb}}_T)> \delta_T+ a_5 \sqrt{n_x}\delta_T)& \leq \Pb(\diam(\Theta_T(w_{\max}+\delta_T))> \delta_T+ a_5 \sqrt{n_x}\delta_T),\\
	& \leq 2\epsilon.
\end{align*}
Therefore, with probability $1-2\epsilon$, the diameter of $\hat \Theta^{\textup{ucb}}_T$ is bounded above by
$$	\diam(\hat \Theta^{\textup{ucb}}_T)\leq \delta_T+ a_5 \sqrt{n_x}\delta_T=O(\sqrt{n_x} \delta_T).$$

Finally, we can verify that the event $\{\diam(\hat \Theta^{\textup{ucb}}_T)\leq \delta_T+ a_5 \sqrt{n_x}\delta_T=O(\sqrt{n_x} \delta_T)\}$ and the event $\{\theta^* \in \hat \Theta^{\textup{ucb}}_T\}$ simultaneously happen with probability at least $1-3\epsilon$ as follows:
\begin{align*}
&\Pb\left(	\theta^*\not \in \hat \Theta_T(\whatmax), \;\text{ or }\; \diam(\hat\Theta_T(w_{\max}+\delta_T))> \delta_T+ a_5 \sqrt{n_x}\delta_T\right)\\
& \leq \Pb\left(w_{\max}\!-\!\max_{0\leq t \leq T-1}\|w_t\|_\infty \geq \delta_T/2 ,\;\text{ or }\;\diam(\Theta_T)>\delta_T/2b_z, \;\text{ or } \;\diam(\hat\Theta_T(w_{\max}+\delta_T))> \delta_T\!+\! a_5 \sqrt{n_x}\delta_T\right)\\
&\leq \Pb\left(w_{\max}\!-\!\max_{0\leq t \leq T-1}\|w_t\|_\infty \geq \delta_T/2\right)+	\Pb\left(\diam(\Theta_T)>\delta_T/2b_z,\; \text{ or }\; \diam(\hat\Theta_T(w_{\max}+\delta_T))> \delta_T\!+\! a_5 \sqrt{n_x}\delta_T\right)\\
& \leq \epsilon+ \Pb(\Ec_2)+\sum_{i=1}^{v_\gamma}\Pb\left(\bigcap_{k} G_{i,k}(\min(\delta_T/2b_z, \delta_T))\right)\\
&\leq 3\epsilon.
\end{align*}
The third inequality follows from 
\begin{itemize}
    \item the proof of Theorem~\ref{thm: convergence rate of wbarmax} in \cref{append: conv rate wbarmax},
    \item Theorem \ref{thm: loose bound}, 
    \item the fact that the probabilties $\Pb(\diam(\Theta_T)>\delta_T/2b_z )$ and $\Pb(\diam(\hat\Theta_T(w_{\max}+\delta_T))> \delta_T+ a_5 \sqrt{n_x}\delta_T)$ are bounded by the same events $\Ec_2$, 
    \item and $ G_{i,k}(\delta_T), G_{i,k}(\delta_T/2b_z) \subseteq G_{i,k}(\min(\delta_T/2b_z, \delta_T))$, where $G_{i,k}(\delta)$ is defined in Lemma \ref{lem: construct Gik and Lik} as a function of $\delta$.
\end{itemize}
This completes the proof. \qed

\section{Simulation details and additional experiments}\label{append: simulation}
This section provides the details on the simulation experiments, along with some additional
results. The code for replicating the presented results can be found in the {github repository}: \url{https://github.com/jy-cds/non-asymptotic-set-membership}.

\subsection{Baseline: LSE's confidence regions}\label{append: LSE computation}
In all our experiments, we use the 90\% confidence region of the LSE as the baseline uncertainty set.  The diameters of LSE's confidence regions are computed by taking minimum of the formulas provided in the following two papers: Lemma E.3 in \cite{simchowitz2020naive} and Theorem 1 in \cite{abbasi2011regret}. 
To apply Theorem 1 in \cite{abbasi2011regret}, we used regularization parameter $\lambda = 0.1$, $\delta=0.1$ for 90\% confidence, $S = \sqrt{\tr \left(\theta^{*,\top} \theta^*\right)}$, variance proxy $L = 1$ for truncated Gaussian distribution and $L= 4/3$ for uniform distribution.

To determine the parameters in Lemma E.3 of \cite{simchowitz2020naive}, we approximately optimize the projection matrix $P$ in Lemma E.3 as follows. First, we consider an orthogonal transformation of the empirical
covariance matrix $\Lambda = \sum_{t=1}^T z_t z_t^\top$ with $\Lambda = GMG^\top$ where $G$ is unitary. This transforms the event $\mathcal{E}$ in Lemma E.3 to $M \geq \lambda_1 P_0 + \lambda_2 (I-P_0)$, where $GP_0 G ^\top = P$. We select $P_0$ as a block matrix $[[I_p,\, 0],[0,\,0]]$, then optimize over the block size $p$ in search of the tightest LSE confidence bound.

\subsection{Figure \ref{fig: teaser}: SME and LSE uncertainty set visualization}

In this experiment, we consider $x_{t+1}=A^*x_t +B^* u_t +w_t$, where $A^*=0.8$ and $B^*=1$ are unknown. $w_t \sim \text{TrunGauss}(0,\sigma_w, [-w_{\max}, w_{\max}])$ is i.i.d. and $u_t \sim\text{TrunGauss}(0,\sigma_u, [-u_{\max}, u_{\max}])$ are also i.i.d generated, where {$\sigma_w=\sigma_u=0.5$, and $w_{\max}=u_{\max}=1$.} We compare SME that knows $w_{\max}=1$ and LSE's 90\% confidence region computed based on  \Cref{append: LSE computation}.

\subsection{Figure \ref{fig:loose_bound}}
In this experiment, we consider the  
 the linearized longitudinal
flight control dynamics of Boeing 747 \cite{lale2022reinforcement, mete2022augmented} with i.i.d.\ bounded inputs and disturbances sampled from truncated Gaussian  and uniform distribution. The dynamics is $x_{t+1} = A^*x_t + B^*u_t + w_t$ with
\begin{equation*}
    A =  \begin{bmatrix}
        0.99 & 0.03 & -0.02& -0.32\\
0.01& 0.47 &4.7 &0\\
0.02 & -0.06 &0.4& 0\\
0.01 & -0.04 &0.72 &0.99
    \end{bmatrix}\quad B = \begin{bmatrix}
        0.01 &0.99\\
-3.44 &1.66\\
-0.83& 0.44\\
-0.47& 0.25
    \end{bmatrix}.
\end{equation*}

Disturbances are sampled from $\text{TrunGauss}(0,I, [-w_{\max}, w_{\max}]^4)$ as well as $\text{Unif}([-w_{\max}, w_{\max}]^4)$, while control inputs are samples from $\text{TrunGauss}(0,I, [-w_{\max}, w_{\max}]^2)$ in both disturbance settings, with $w_{\max} = 2$. To compute the UCB for SME using \eqref{equ: ucb-sme}, we heuristically define $\delta_T = \beta \, \frac{ n_x^{1.5} \cdot n_z^{2} \cdot (\max_t \|x_t\|)}{T}$, where $n_x = 4$ and $n_z = 6$ are the system dimension, while $\beta$ is a tunable parameter. This definition matches the dimension and time order of the theoretical analysis in \Cref{cor: wbarmax bdd when qw=O(epsilon)}. In both experiments of \Cref{fig:loose_bound}, we fix $\beta = 0.01$.

In Figure 2(a)-(b), we plot SME with accurate and conservative bounds of  $w_{\max}$, UCB-SME, and LSE's 90\% confidence regions computed by \Cref{append: LSE computation}. We use 10 different seeds to generate the disturbance sequences for each plot, and use the shaded region to denote 1 standard deviation from the mean (colored lines).

\subsection{Figure \ref{fig:dimension_gaussian}}
In this experiment, we consider autonomous systems of the form $x_{t+1} =A^*x_t +w_t$, where $A^*\in\mathbb{R}^{n_x}$ is randomly sampled and its spectral radius is normalized to be $0.9$. We simulate SME and LSE for $n_x=5,10,15,20,25$. The disturbances are sampled from $\text{TrunGauss}(0,I, [-w_{\max}, w_{\max}]^{n_x})$ as well as $\text{Unif}([-w_{\max}, w_{\max}]^{n_x})$ with $w_{\max} = 2$. This simula!tion is run on 10 random seeds and the total length of the simulation is set to be $T=1000$ across all $n_x$ experiments. The mean is plotted as solid lines and the shaded regions denote 1 standard deviation from the mean.

Though SME's theoretical bound with respect to the dimension is  $\tilde O(n_x^{1.5} n_z^2 )$ from \Cref{cor: estimation err bdd for B=0}, which is much worse than LSE's bound, it is not reflected in \Cref{fig:dimension_gaussian}. Therefore, it is promising that the dimension scaling in the analysis in \Cref{sec:analysis} can be further tightened. We leave this for future work.

\subsection{Figure \ref{fig:ra-mpc}}
To illustrate the quantitative impact of using SME for adaptive tube-based robust MPC, we study tube-based robust MPC for a system $x_{t+1}=A^*x_t+B^* u_t+w_t$ with nominal system $A^* = 1.2$, $B^* = 0.9$ with an initial model uncertainty set $\Theta_0 := [1,1.2]\times[0.9,1.1]$. We use the basic tube-based robust MPC method \cite{rawlings2017model,mayne2005robust} and parameterize the control policy as $u_k = K x_k + v_k + \eta_k$, where $K=-1$, $v_k$ is determined by the tube-based robust MPC algorithm,  and $\eta_k$ is a bounded exploration injection with $\eta_k \sim \textup{Unif}([-0.01, 0.01])$. The disturbance $w_k$ has a known bound of $w_{\max} = 0.1$ and is generated to be i.i.d. $\textup{Unif}([-0.1, 0.1])$. The horizon of  the tube-based robust MPC is set to be 5. The state and input constraints are such that $x_k \in [-10,10] $ and $u_k \in [-10,10]$ for all $k \geq 0$. We consider the task of constrained
LQ tracking problem with a time-varying cost function $c_t := (x_t - g_t)^\top Q (x_t {-} g_t) + u_t^\top R u_t$ where the target trajectory is generated as $g_t = 8 \sin(t/20)$. 

We compare the performance of an adaptive tube-based robust MPC controller that uses the SME for uncertainty set estimation against one that uses the LSE 90\% confidence region (LSE). For better visualization of the trajectory difference as a result of different estimation methods, we used the minimum of the the dominant factors in \citet[equation C.12]{dean2018regret} and the LSE 90\% confidence region for the LSE uncertainty set.
 We also plot the offline optimal RMPC controller, i.e., the controller that has knowledge of the true underlying system parameters (OPT).

Since the controller has to robustly satisfy constraints against the worst-case model in the uncertainty set, smaller uncertainty set for the tube-based robust MPC means more optimal trajectories can be computed. This observation is consistent with the extensive empirical results in the control literature \cite{lorenzen2019robust, lu2019robust,kohler2019linear}.

\end{document}